\documentclass[a4paper,11pt,reqno]{amsart}
\usepackage[normalem]{ulem}
\usepackage{cancel}
\usepackage{amsmath,amssymb,amsthm}
\usepackage{latexsym}
\usepackage{color}
\usepackage{graphicx}
\usepackage{mathrsfs}
\usepackage{enumerate}
\usepackage{bm}
\usepackage[abbrev]{amsrefs}
\usepackage[T1]{fontenc}
\usepackage{mathtools}
\mathtoolsset{showonlyrefs=true}
\usepackage[colorlinks,
            linkcolor=blue,      
            anchorcolor=blue,  
            citecolor=red,       
            ]{hyperref}

\allowdisplaybreaks[4]

\theoremstyle{plain}
\newtheorem{thm}{Theorem}[section]
\newtheorem{lemm}[thm]{Lemma}
\newtheorem{prop}[thm]{Proposition}

\theoremstyle{definition}
\newtheorem{df}[thm]{Definition}
\newtheorem{rem}[thm]{Remark}

\makeatletter

\@addtoreset{equation}{section}
\makeatother

\newcommand{\sgn}{\operatorname{sgn}}
\renewcommand{\div}{\operatorname{div}}
\newcommand{\dB}{\dot{B}}
\newcommand{\dH}{\dot{H}}

\newcommand{\supp}{\operatorname{supp}}

\renewcommand{\leq}{\leqslant}
\renewcommand{\geq}{\geqslant}

\newcommand{\xih}{\xi_{\rm h}}
\newcommand{\xh}{x_{\rm h}}
\newcommand{\yh}{y_{\rm h}}

\newcommand{\R}{\mathbb{R}}

\newcommand{\N}{\mathbb{N}}
\newcommand{\Z}{\mathbb{Z}}

\newcommand{\n}[1]{{\left\|#1\right\|}}

\newcommand{\abso}[1]{{\left|#1\right|}}
\newcommand{\lp}[1]{\left[#1\right]}
\newcommand{\Mp}[1]{\left\{#1\right\}}
\renewcommand{\sp}[1]{\left(#1\right)}
\newcommand{\f}{\frac}
\newcommand{\Om}{\Omega}

\begin{document}
\title[The rotating Navier--Stokes equations]
{Sharp decay estimates for global solutions to the incompressible rotating Navier--Stokes equations}
\author[M.~Fujii]{Mikihiro Fujii}
\address[M.~Fujii]{Graduate School of Science, Nagoya City University}
\email[M.~Fujii]{fujii.mikihiro@nsc.nagoya-cu.ac.jp}
\author[Y.~Li]{Yang Li}
\address[Y.~Li]{School of Mathematical Sciences, Anhui University, Hefei, 230601, People's Republic of China}
\email[Y.~Li]{lynjum@163.com}
\author[J.~Xu]{Jiang Xu}
\address[J.~Xu]{School of Mathematics, Nanjing University of Aeronautics and Astronautics, Nanjing 211106, People's Republic of China}
\email[J.~Xu]{jiangxu\underline{~}79@nuaa.edu.cn}
\begin{abstract}
In this paper, we consider the three-dimensional incompressible rotating Navier--Stokes equations and establish the sharp $L^p$ decay estimates of global solutions. We reveal that the optimal $L^p$ decay rates for $2<p<\infty$ are strictly faster than those obtained in existing results by interpolation between the $L^2$ unitary identity and $L^\infty$ dispersive estimates, although the  endpoint cases were known to be sharp. Moreover, the optimality of decay rates is also proved by the lower bound estimate for a specific initial datum.
The underlying mechanism lies in the anisotropic degeneracy of the oscillatory integrals arising from the Coriolis force.
\end{abstract}
\maketitle

\tableofcontents

\section{Introduction}
In this paper, we consider the initial value problem for the three-dimensional incompressible rotating Navier--Stokes equations: 
\begin{align}\label{eq:nonlin}
    \begin{cases}
        \partial_t u - \Delta u + \Omega e_3 \times u + (u \cdot \nabla)u + \nabla P = 0, \qquad & t>0,x \in \R^3, \\
        \div u = 0, & t \geq 0, x \in \R^3, \\
        u(0,x)=u_0(x), & x \in \R^3.
    \end{cases}
\end{align}
Here, $u=u(t,x):[0,\infty) \times \R^3 \to \R^3$ and $P=P(t,x):(0,\infty) \times \R^3 \to \R$ stand for the unknown velocity and pressure of the fluid, respectively, whereas $u_0=u_0(x):\R^3 \to \R^3$ is the given initial velocity satisfying the compatibility condition $\div u_0=0$.
For the Coriolis term $\Omega e_3 \times u$, 
the constant $\Om \in \R$ represents the angular speed of the rotation around the vertical direction $e_3=(0,0,1)$. 
The mathematical analysis of \eqref{eq:nonlin} was initiated by Babin--Mahalov--Nicolaenko \cites{Bab-Mah-Nic-97,Bab-Mah-Nic-99,Bab-Mah-Nic-01}, and has subsequently been developed in a large body of work. For the global well-posedness, we refer the readers to \cites{Iwa-Tak-13,Iwa-Tak-14,Koh-Lee-Tak-14,CDGG,CDGG2,CDGG3,Fuj-24}; for the fast rotation limit, to \cites{Ohy-Tak-21,Fuj-26,CDGG,CDGG2,CDGG3}; and for the large time behavior, to \cites{Ahn-Kim-Lee-22,Ega-Tak-23,Yos-26,Kim-22}.

Applying the Helmholtz projection $\mathbb{P}=I+\nabla (-\Delta)^{-1}\div $ to the first equation of \eqref{eq:nonlin} and using the Duhamel principle, we may rewrite \eqref{eq:nonlin} as the integral equation
\begin{align}
    u(t)
    = 
    \mathcal{T}_\Om(t)u_0
    -
    \int_0^t
    \mathcal{T}_\Om(t-\tau)\mathbb{P}\div (u(\tau) \otimes u(\tau))\, d\tau,
\end{align}
where $\{\mathcal{T}_\Om(t)\}_{t>0}$ is the semigroup generated by the Stokes--Coriolis operator $-\Delta+\Om \mathbb P e_3 \times \mathbb{P}$.
From the work of Hieber--Shibata \cite{Hie-Shi-10}, it is known that this semigroup is represented explicitly as 
\begin{align}
    \mathcal{T}_\Omega(t)u_0
    :={}&
    e^{t(\Delta - \Om \mathbb Pe_3 \times \mathbb{P})}u_0
    \\
    ={}&
    e^{t\Delta}e^{it\Omega \frac{D_3}{|D|}}P_+(D) u_0
    +
    e^{t\Delta}e^{-it\Omega \frac{D_3}{|D|}}P_-(D) u_0.
\end{align}
where the Fourier multiplier $P_\pm(D)=\mathscr{F}^{-1} P_\pm(\xi)\mathscr{F}$ is defined via 
\begin{align}\label{def-of-PR}
    P_\pm(\xi) = \frac{1}{2}\sp{I \pm iR(\xi)},
    \qquad
    R(\xi)
    :=
    \frac{1}{|\xi|}
    \begin{pmatrix}
        0 &  -  \xi_3 & \xi_2 \\
        \xi_3 & 0 & -  \xi_1 \\
        - \xi_2 & \xi_1 & 0
    \end{pmatrix}. 
\end{align}
One of the key features distinguishing the rotating Navier--Stokes equations from the classical Navier--Stokes equations is that their linearized semigroup includes the dispersive evolution group $\{e^{i\tau \frac{D_3}{|D|}}\}_{\tau \in \R}$ generated by the Coriolis force.
To see the dispersive nature, let us focus on the Fourier integral
\begin{align}
    e^{i\tau \frac{D_3}{|D|}}\varphi(x)
    =
    \frac{1}{(2\pi)^3}
    \int_{\R^3}
    e^{ix\cdot \xi}
    e^{i\tau \frac{\xi_3}{|\xi|}}
    \widehat \varphi(\xi)\, d\xi,
\end{align}
where $\varphi \in \mathscr{S}(\R^3)$ is a given frequency amplitude satisfying $\supp \widehat{\varphi} \subset \{ \xi \in \R^3\ ;\ 1/4 \leq |\xi| \leq 4\}$.
In \cites{Koh-Lee-Tak-14,Koh-Lee-Tak-14-JDE}, the following dispersive decay estimate was shown:
\begin{align}\label{L^2:disp}
    \n{e^{i\tau \frac{D_3}{|D|}}\varphi}_{L^\infty}
    \leq 
    C(1+|\tau|)^{-1},
    \qquad \tau \in \R,
\end{align}
with some positive constant $C$ depending only on $\varphi$.
We emphasize that, due to the classical Littman lemma, the decay rate $-1$ comes from the following degeneracy of the phase function
\begin{align}
    \operatorname{rank}
    \nabla^2 \sp{\frac{\xi_3}{|\xi|}} \geq 2
    \qquad \text{on } \supp \widehat{\varphi}
    .
\end{align}
We also remark that Koh--Lee--Takada \cite{Koh-Lee-Tak-14-JDE} proved the polynomial decay order $-1$ is sharp.
On the other hand, the Plancherel theorem immediately implies 
\begin{align}\label{L^infty:disp}
    \n{e^{i\tau \frac{D_3}{|D|}}\varphi}_{L^2}=\n{\varphi}_{L^2},
    \qquad \tau \in \R.
\end{align}
Then, by employing the interpolation between \eqref{L^2:disp} and \eqref{L^infty:disp}, the $L^p$ decay rates of the group $e^{i\tau \frac{D_3}{|D|}}\varphi$ is given by
\begin{align}\label{decay-inter}
    \n{e^{i\tau \frac{D_3}{|D|}}\varphi}_{L^p}
    \leq 
    C(1+|\tau|)^{-(1-\frac{2}{p})},
    \qquad \tau \in \R,\ 2 \leq p \leq \infty.
\end{align}
Taking the dispersive effect into account of the large time analysis for rotating Navier--Stokes flow,  the linear $L^1$-$L^p$ type decay estimate was deduced by \cite{Ega-Tak-23} as follows:
\begin{align}\label{decay:previ}
    \n{\nabla^m e^{t\Delta}e^{i\tau \frac{D_3}{|D|}}f}_{L^p}
    \leq 
    Ct^{-\frac{3}{2}(1-\frac{1}{p})-\frac{m}{2}}(1+|\tau|)^{-(1-\frac{2}{p})}
    \n{f}_{L^1}, \qquad t>0,\ \tau \in \R
\end{align}
for $2 \leq p \leq \infty$ and $m \in \N \cup \{0\}$.
Furthermore, it was proved by \cites{Ega-Tak-23,Yos-26} that the nonlinear solution $u$ with the initial datum in the class $u_0 \in L^1(\R^3) \cap \dH^s(\R^3)$ ($1/2<s<9/10$) 
satisfies for every $2 \leq p \leq \infty$ and $m \in \N \cup \{0\}$ that
\begin{align}
    \n{\nabla^m u(t)}_{L^p}
    =
    o\sp{t^{-\frac{3}{2}(1-\frac{1}{p})-\frac{m}{2}}(1+|\Om|t)^{-(1-\frac{2}{p})}}
\end{align}
as $t \to \infty$, and $1/2$-enhanced decay estimates
\begin{align}
    \n{\nabla^m u(t)}_{L^p}
    \leq 
    C
    t^{-\frac{3}{2}(1-\frac{1}{p})-\frac{m+1}{2}}
    (1+|\Om|t)^{-(1-\frac{2}{p})}, \qquad t>0,
\end{align}
provided that the additional condition $|x|u_0(x) \in L^1(\R^3)$ is imposed. Observe that the asymptotic expansion of Fujigaki--Miyakawa \cite{Fuj-Miya-01} type for some $2<p<3$ was established in  \cite{Ega-Tak-23}. We refer to \cites{Ahn-Kim-Lee-22,Kim-22} for the temporal decay estimates of solutions to the rotating fractional Navier--Stokes equations and rotating MHD equations.

However, it remains an open problem whether the additional decay rate $(1+|\Om|t)^{-(1-\frac{2}{p})}$ induced by the dispersive effects of the Coriolis force is optimal except for the endpoint cases $p=2,\infty$. Let us stress that the optimality when $p=2$ is classical in incompressible Navier--Stokes equations and the optimality when $p=\infty$ was verified in Koh--Lee--Takada \cite{Koh-Lee-Tak-14-JDE}. 
\emph{In this paper, we reveal that the dispersive decay rates $(1+|\Om|t)^{-(1-\frac{2}{p})}$ obtained by interpolation are in fact not optimal, and that the true optimal decay rates are strictly faster for $2<p<\infty$}.
More precisely, we prove in Section \ref{sec:lin} that for $m \in \N \cup \{0\}$, 
\begin{align}
    \n{\nabla^m
    e^{t\Delta}e^{i\tau \frac{D_3}{|D|}}f}_{L^p}
    &\leq 
    \begin{cases}
    Ct^{-\frac{3}{2}(1-\frac{1}{p})-\frac{m}{2}}
    (1+|\tau|)^{-\frac{3}{2}(1-\frac{2}{p})}\n{f}_{L^1} & (2 \leq p < 4),
    \\
    Ct^{-\frac{3}{2}(1-\frac{1}{p})-\frac{m}{2}}
    (1+|\tau|)^{-(1-\frac{1}{p})}\n{f}_{L^1} & (4 < p \leq \infty)
    \end{cases}
\end{align}
for $f \in L^1(\R^3)$, which is better than \eqref{decay:previ}.
Moreover, for the case of $p=4$, we show that
\begin{align}
    \n{\nabla^m
    e^{t\Delta}e^{i\tau \frac{D_3}{|D|}}f}_{L^4}
    \leq 
    Ct^{-\frac{3}{2}(1-\frac{1}{4})-\frac{m}{2}}
    (1+|\tau|)^{-\frac{3}{4}}(\log(e+|\tau|))^{\frac{1}{4}}\n{f}_{L^1},
\end{align}
and the logarithmic correction may be removed by using the weak Lebesgue space, that is, 
\begin{align}
    \n{\nabla^m
    e^{t\Delta}e^{i\tau \frac{D_3}{|D|}}f}_{L^{4,\infty}}
    \leq 
    Ct^{-\frac{3}{2}(1-\frac{1}{4})-\frac{m}{2}}
    (1+|\tau|)^{-\frac{3}{4}}\n{f}_{L^1}.
\end{align}
\emph{Observe that our faster decay rates are shown to be sharp since it is possible to establish the lower bound estimates}; see Theorem \ref{Thm:sharp} below for the related topic. 
This is another novelty of the current work, since as far as we know there seems to be no previous results on the lower decay estimates for incompressible rotating Navier--Stokes equations.

Applying the above linear analysis to the nonlinear solutions with critical Sobolev regularity, we obtain the first main result of this paper.
\begin{thm}\label{main-thm-1}
    Let $2\leq p\leq \infty$ and $m \in \N \cup \{0\}$.
    Let $u_0 \in \dot H^{\frac{1}{2}}(\R^3)\cap L^1(\R^3)$ satisfy $\div u_0 = 0$.
    Then, there exists a positive constant $\Omega_0=\Omega_0(u_0)$ such that
    for any $\Omega \in \R$ with $|\Omega| \geq \Omega_0$, \eqref{eq:nonlin} possesses a unique mild solution $u \in C([0,\infty);H^{\frac{1}{2}}(\R^3)) \cap C((0,\infty);L^\infty(\R^3))$ and it holds 
    \begin{align}\label{non-dec-L^p}
        \n{\nabla^m u(t)}_{L^p} = 
        \begin{cases}
            o\sp{t^{-\frac{3}{2}(1-\frac{1}{p})-\frac{m}{2}}(1+|\Omega|t)^{-\frac{3}{2}(1-\frac{2}{p})}} & (2 \leq p <4),\\
            o\sp{t^{-\frac{3}{2}(1-\frac{1}{4})-\frac{m}{2}}(1+|\Om|t)^{-\frac{3}{4}}\sp{\log(e+|\Omega|t)}^{\frac{1}{4}}} & (p = 4),\\
            o\sp{t^{-\frac{3}{2}(1-\frac{1}{p})-\frac{m}{2}}(1+|\Omega|t)^{-(1-\frac{1}{p})}} & (4<p\leq \infty)
        \end{cases}
    \end{align}
    as $t \to \infty$.
\end{thm}
As for the classical Navier--Stokes equations, the small order decay rate may be improved to the additional $t^{-\frac{1}{2}}$-faster decay estimate, provided that $|x|u_0(x)$ is integrable.
In the following theorem, we state not only this but also removal of the logarithmic correction for $p=4$.
\begin{thm}\label{thm:main-2}
    Let $u_0 \in \dot H^{\frac{1}{2}}(\R^3)\cap L^1(\R^3)$ satisfy $\div u_0 = 0$ and additionally assume $|x|u_0(x) \in L^1(\R^3_x)$.
    Then, for every $2 \leq p \leq \infty$ and $m \in \N \cup \{0\}$, 
    there exists a positive constant $C=C(m,p,u_0)$ such that the solution $u$ to \eqref{eq:nonlin} with $|\Omega| \geq \Omega_0$, constructed in Theorem \ref{main-thm-1}, satisfies the $1/2$-enhanced decay estimates
    \begin{align}\label{enh-non-dec-L^p} 
        &
        \n{\nabla^mu(t)}_{L^p}
        \leq
        \begin{cases}
            Ct^{-\frac{3}{2}(1-\frac{1}{p})-\frac{m+1}{2}}(1+|\Omega|t)^{-\frac{3}{2}(1-\frac{2}{p})} & (2 \leq p <4),\\
            Ct^{-\frac{3}{2}(1-\frac{1}{4})-\frac{m+1}{2}}(1+|\Om|t)^{-\frac{3}{4}}\sp{\log(e+|\Omega|t)}^{\frac{1}{4}} & (p = 4),\\
            Ct^{-\frac{3}{2}(1-\frac{1}{p})-\frac{m+1}{2}}(1+|\Omega|t)^{-(1-\frac{1}{p})} & (4<p\leq \infty)\\
        \end{cases}
    \end{align}
    for all $t \geq 1$. Furthermore, the logarithmic correction for $p=4$ may be removed in $L^{4,\infty}$ framework:
    \begin{align}\label{corrc-4}
        \n{\nabla^mu(t)}_{L^{4,\infty}}
        \leq 
        C
        t^{-\frac{3}{2}(1-\frac{1}{4})-\frac{m+1}{2}}
        (1+|\Om|t)^{-\frac{3}{4}}.
    \end{align}
\end{thm}
\begin{rem}
    For the threshold case $p=4$, the strength of logarithmic correction is actually controlled by the interpolation in Lorentz spaces:
    \begin{align}
        \n{u(t)}_{L^{4,q}}
        \leq 
        Ct^{-\frac{3}{2}(1-\frac{1}{4})-\frac{m+1}{2}}
        (1+|\Om|t)^{-\frac{3}{4}}
        \sp{\log(e+|\Omega|t)}^{\frac{1}{q}},
        \qquad
        1 \leq q \leq \infty,
    \end{align}
    which can be easily obtained by slightly modifying the linear analysis.
\end{rem}
Next, we claim that the decay estimates \eqref{enh-non-dec-L^p} and \eqref{corrc-4} are optimal.
\begin{thm}\label{Thm:sharp}
    Let $\Omega = 1$ and $2\leq p \leq \infty$.
    Then, there exist $u_0 \in \mathscr{S}(\R^3)$ with $\div u_0=0$ and positive constants $c=c(p)$, $T=T(p)$ such that the associated mild solution $u$ to \eqref{eq:nonlin} satisfies 
    \begin{align}
        &
        \n{u(t)}_{L^p}
        \geq 
        \begin{cases}
            ct^{-\frac{3}{2}(1-\frac{1}{p})-\frac{1}{2}}(1+t)^{-\frac{3}{2}(1-\frac{2}{p})} & (2 \leq p <4),\\
            ct^{-\frac{3}{2}(1-\frac{1}{4})-\frac{1}{2}}(1+t)^{-\frac{3}{4}}\sp{\log(e+t)}^{\frac{1}{4}} & (p = 4),\\
            ct^{-\frac{3}{2}(1-\frac{1}{p})-\frac{1}{2}}(1+t)^{-(1-\frac{1}{p})} & (4<p\leq \infty),\\
        \end{cases}
        \\
        &
        \n{u(t)}_{L^{4,\infty}}
        \geq 
        c
        t^{-\frac{3}{2}(1-\frac{1}{4})-\frac{1}{2}}
        (1+t)^{-\frac{3}{4}}
    \end{align}
    for all $t \geq T$. 
\end{thm}
\begin{rem}
    In Theorem \ref{Thm:sharp}, the initial datum 
    is chosen explicitly by 
    \begin{align}
        u_0(x):=\varepsilon e^{-\frac{|x|^2}{4}}(x_2,-x_1,0),
    \end{align}
    where $\varepsilon>0$ is chosen sufficiently small; see Theorem \ref{Thm:sharp-lin}.
\end{rem}

We briefly explain why the $L^p$ decay estimate \eqref{decay-inter} derived by the interpolation does not capture the true behavior for
$2<p<\infty$.  
As mentioned above, the oscillatory decay rate is controlled by the degeneracy of the phase $\xi_3/|\xi|$ and it is measured by the Hessian
\[
h(\xi):=\det\nabla_\xi^2\sp{\frac{\xi_3}{|\xi|}}
=\frac{|\xi_h|^2\xi_3}{|\xi|^9}.
\]
Along this function, we decompose the oscillatory kernel into
dyadic degeneracy layers
\begin{align}
    e^{i\tau \frac{D_3}{|D|}}\varphi=\sum_{k\leq k_0}I_k(\tau),
    \qquad
    I_k(\tau,x)
    =
    \frac{1}{(2\pi)^3}
    \int_{|h(\xi)|\simeq2^k}
    e^{ix \cdot \xi}
    e^{i\tau \frac{\xi_3}{|\xi|}}
    \widehat \varphi (\xi)\, d\xi.
\end{align}
By the estimate
\begin{align*}
\left|
\Mp{\xi \in \supp \widehat \varphi\ ;\ |h(\xi)|\simeq 2^k}
\right| 
\simeq
\left|
\Mp{\xi \in \supp \widehat \varphi\ ;\ |\xih|^2 |\xi_3|\simeq 2^k}
\right| 
\simeq 2^k,
\end{align*} 
and the standard stationary phase estimate, it is naturally expected that the following estimate holds:
\begin{align}
    \|I_k(\tau)\|_{L^\infty}
    \leq 
    C\min\Mp{2^k,\,|\tau|^{-\frac{3}{2}}2^{-\frac{k}{2}}},
    \qquad
    \|I_k(\tau)\|_{L^2}\leq C 2^{k/2}.
\end{align}
Interpolating these estimates yields
\begin{align}\label{I_k-L^p}
\|I_k(\tau)\|_{L^p}
\leq 
C
\min\left\{
2^{(1-\frac{1}{p})k},\,
|\tau|^{-\frac32(1-\frac2p)}
2^{(\frac2p-\frac12)k}
\right\},
\qquad 
2\leq p\leq\infty.
\end{align}
The two bounds balance at $2^k\simeq|\tau|^{-1}$.  The global
$L^\infty$ estimate is therefore governed by this transition layer and
has size $|\tau|^{-1}$.  Interpolating only this worst $L^\infty$ bound, which provides \eqref{L^2:disp},
with the conserved $L^2$ norm \eqref{L^infty:disp}, we treat the oscillatory integral as if it were as degenerate as the transition layer throughout the whole annulus $\supp \widehat \varphi$, and this is precisely why the known decay rate
$|\tau|^{-(1-\frac{2}{p})}$ in \eqref{decay-inter} is not sharp.
For $2\leq p<4$, the exponent $2/p-1/2$ in \eqref{I_k-L^p} is positive, so the sum over the
less degenerate layers is dominated by the nondegenerate part of the
annulus and gives
$|\tau|^{-\frac32(1-\frac2p)}$; the more degenerate layers are smaller
because of their shrinking Fourier volume.  For $4<p\leq \infty$, the
same sum is concentrated near the transition scale
$2^k\simeq|\tau|^{-1}$ and gives $|\tau|^{-(1-\frac{1}{p})}$.  
Thus, the sharp $L^p$ decay structure with $2<p<\infty$ includes not only the worst pointwise degeneracy but
also the small size of the frequency region where that degeneracy of the phase
occurs, and the latter effect contributes the faster decay. 
\emph{This type of phenomenon does not appear in the isotropic case (see for instance \cite{Guo-Peng-Wang-08}); rather, it is caused by the anisotropy of the phase function in the oscillatory integral. To the best of our knowledge, the present paper is the first one to identify and capture such a phenomenon.}

The rest of this paper is arranged as follows. 
In Section \ref{sec:pre}, we collect some preliminary material, including some basic facts about Besov spaces and some useful lemmas. In Section \ref{sec:lin}, we proceed beyond the above heuristic argument and develop a unified framework combining an improved analysis of oscillatory integrals with distribution function estimates, which may even provide the optimal logarithmic correction in $L^4$ estimate and also give $L^{4,\infty}$ estimate. In Section~\ref{sec:non}, we apply those sharp linear estimates to the nonlinear problem and prove the main results.

\section{Preliminaries}\label{sec:pre}
In this section, we explain some notations and correct lemmas that will be used throughout this paper.
\subsection{Littlewood--Paley theory and function spaces}
In this subsection, we recall some function spaces which will be used in the sequel. 
Let $\mathscr{S}(\R^3)$ be the space of all Schwartz functions on $\R^3$ and $\mathscr{S}'(\R^3)$ be the set of all tempered distributions on $\R^3$. For any $f\in \mathscr{S}(\R^3)$, we define the Fourier transform and inverse Fourier transform by
\begin{align}
\mathscr{F}[f](\xi)=\widehat{f}(\xi)=
\int_{\R^3} e^{- i x\cdot \xi} f(x)\,dx, \qquad
\mathscr{F}^{-1}[f](x)= (2\pi)^{-3} \int_{\R^3} e^{ i x\cdot \xi} f(\xi)\,d\xi.
\end{align} 
We recall the definition of the Littlewood--Paley decomposition and homogeneous Besov spaces.
Let $\phi \in C_c^{\infty}([0,\infty];[0,1])$  satisfy
$\supp \phi \subset [2^{-1},2]$ and $\sum_{j \in \Z} \phi(2^{-j}r) = 1$ for all $r>0$.
Let us define the dyadic localization operators $\{\Delta_j\}_{j \in \Z}$ by
\begin{align}
    \Delta_j f := \mathscr{F}^{-1}\lp{\widehat\varphi_j(\xi)\widehat{f}(\xi)},
    \qquad
    \widehat {\varphi_j}(\xi):=\phi(2^{-j}|\xi|).
\end{align}
For any $s\in \R,1\leq p,q\leq \infty$, we define the homogeneous Besov space $\dB^{s}_{p,q}(\R^3)$ by 
\begin{align}
 \dB^{s}_{p,q}(\R^3)&
 := \left\{  f \in \mathscr{S}'(\R^3)/\mathscr{P}(\R^3) ;\,
\n{f}_{  \dB^{s}_{p,q} } <\infty 
\right\}, \\
\n{f}_{  \dB^{s}_{p,q} } & :=
\n{
\Mp{ 2^{sj} \n{\Delta_j f}_{L^p} }_{j \in \Z}
}_{\ell^{q}(\Z)}, 
\end{align} 
where $\mathscr{P}(\R^3)$ denotes the set of all polynomials on $\R^3$. 
The following embeddings are well-known:
\begin{align}
    \dB_{p,1}^0(\R^3) \hookrightarrow L^p(\R^3) \hookrightarrow \dB_{p,\infty}^0(\R^3), \qquad 1 \leq p \leq \infty,
\end{align}
and 
\begin{align}
    \dB_{p_1,q}^{\frac{3}{p_1}} (\R^3) \hookrightarrow \dB_{p_2,q}^{\frac{3}{p_2}} (\R^3),
    \qquad 1 \leq p_1 \leq  p_2 \leq \infty,\ 1 \leq q \leq \infty.
\end{align}
See \cite{Bah-Che-Dan-11} for the details.
In the particular case of $p=q=2$, we may identify $\dB^s_{2,2}(\R^3)$ with the homogeneous Sobolev space $\dH^s(\R^3)$ in the sense of norm equivalence.

We summarize some basic results on Besov spaces. 
We begin with an estimate of heat flow on Besov spaces.
\begin{lemm}[\cite{Koz-Oga-Tan-03}*{Lemma 2.2}]\label{lemm:heat}
    Let $1\leq p \leq \infty$ and $s_1,s_2 \in \R$ with $s_1<s_2$.
    Then, there exists a positive constant $C=C(s_1,s_2)$ such that
    \begin{align}
        \n{e^{t\Delta}f}_{\dB_{p,1}^{s_2}} \leq C t^{-\frac{s_2-s_1}{2}}\n{f}_{\dB_{p,\infty}^{s_1}}
    \end{align}
    for all $t>0$ and $f \in \dB_{p,\infty}^{s_1}(\R^3)$.
\end{lemm}

Next, we recall a real interpolation inequality for Besov spaces.
\begin{lemm}[\cite{Bah-Che-Dan-11}*{Proposition 2.22}]\label{lemm:interp}
    Let $1\leq p \leq \infty$ and $s_0,s_1 \in \R$ with $s_0<s_1$. Let $0<\theta<1$. Then there exists a positive constant $C=C(s_1,s_2,\theta)$ such that
    \begin{align}
        \n{f}_{\dB_{p,1}^{\theta s_1 + (1-\theta)s_0}} 
        \leq 
        C
        \n{f}_{\dB_{p,\infty}^{s_1}}^{\theta}
        \n{f}_{\dB_{p,\infty}^{s_0}}^{1-\theta} 
    \end{align}
    for all $f \in \dB_{p,\infty}^{s_1}(\R^3) \cap \dB_{p,\infty}^{s_2}(\R^3)$. 
\end{lemm}

Let us review the following basic para-product estimate in Besov spaces. 
\begin{lemm}[\cite{Cha-04}*{Lemma 2.2}]\label{lemm:prod-est}
    For $s>0$, $1\leq p,p_1,p_2,q\leq \infty$ with $1/p  = 1/p_1 + 1/p_2$, there exists a positive constant $C=C(s,p,p_1,p_2,q)$ such that 
    \begin{align}
        \n{fg}_{ \dB^{s}_{p,q}}
        \leq 
        C \sp{
        \n{f}_{L^{p_1}} 
        \n{g}_{\dB^{s}_{p_2,q}} 
        + 
        \n{f}_{\dB_{p_2,q}^s}\n{g}_{L^{p_1}} 
        }
    \end{align} 
    for all  $f,g\in L^{p_1}(\R^3) \cap \dB^{s}_{p_2,q}(\R^3)$.
\end{lemm}
We remark that the proof of Lemma \ref{lemm:prod-est} can also be obtained from \cite{Bah-Che-Dan-11}*{Theorem 2.47, 2.52}. The details are omitted.

Finally, we recall the weak Lebesgue spaces. For $1\leq p< \infty$, the weak Lebesgue space $L^{p,\infty}(\R^3)$ consists of all measurable functions $f$ on $\R^3$ such that
\begin{align}
    \n{f}_{L^{p,\infty}}
    :=
    \sup_{\lambda>0} 
    \lambda
    \abso{\{x \in \R^3;\, |f(x)|>\lambda  \}}^\frac{1}{p}
    <\infty, 
\end{align} 
It is well-known that $L^p(\R^3)\subsetneq L^{p,\infty}(\R^3)$ when $1\leq p<\infty$. For more details on the function spaces mentioned above, we refer to \cites{Bah-Che-Dan-11,Gra-14}.

\subsection{Stationary phase lemmas for oscillatory integrals}
Let us prepare the estimates for oscillatory integrals.
\begin{lemm}\label{lemm:sta}
    Let $K \subset \R^3$ be a non-empty compact set.
    Let $\{\psi_\alpha\}_{\alpha \in A} \subset C^6(K;\R)$ and $\{b_\alpha\}_{\alpha \in A} \subset C_c^4(\R^3;\R)$ be families of phase and amplitude functions satisfying
    \begin{align}
        \inf_{\alpha \in A}
        \inf_{\xi \in K}
        \abso{\det \nabla^2 \psi_\alpha(\xi)}>0,
        \qquad
        \bigcup_{\alpha \in A}
        \supp b_\alpha
        \subset K,
    \end{align}
    and
    \begin{align}
        \sum_{m=0}^6
        \sup_{\alpha \in A}
        \n{\nabla^m\psi_\alpha}_{L^\infty(K)}<\infty,
        \qquad
        \sum_{m=0}^4
        \sup_{\alpha \in A}
        \n{\nabla^m b_\alpha}_{L^\infty(K)}<\infty.
    \end{align}
    Then, there exists a positive constant $C$ such that 
    \begin{align}
        \sup_{\alpha \in A}
        \sup_{x \in \R^3}
        \abso{\int_{\R^3}e^{ix\cdot \xi}e^{i\tau \psi_{\alpha}(\xi)}b_{\alpha}(\xi)\, d\xi}
        \leq C|\tau|^{-\frac{3}{2}}
    \end{align}
    for all $\tau \in \R$ with $|\tau| \geq 1$.
\end{lemm}
One may prove Lemma \ref{lemm:sta} along the same argument as the standard stationary phase estimate; see \cite{Hor-90}*{Theorem 7.7.1} or \cite{Ste-Sha-11}*{Proposition 2.5, page 329} 
for instance.
Next, we review the asymptotic expansion of oscillatory integrals.
\begin{lemm}[\cite{Hor-90}*{Theorem 7.7.6}]\label{lemm:asy}
    Let $K \subset \R^3$ be a non-empty compact set. 
    Let $\psi=\psi(y,\xi) \in C^{\infty}(K \times \R^3;\R)$ and $a=a(y,\xi) \in C_c^{\infty}(K \times \R^3)$.
    Assume that there exists a $\Xi \in C^{\infty}(K;\R^3)$ such that 
    \begin{itemize}
        \item 
        For $y \in K$, $\xi=\Xi(y)$ is a unique stationary point for $\R^3 \ni \xi \mapsto \psi(y,\xi)\in \R$ on $\supp a(y,\cdot)$.
        \item 
        It holds $\det \nabla_\xi^2 \psi(y,\Xi(y))\neq 0$ for all $y \in K$.
    \end{itemize}
    Then, there exists a positive constant $C=C(K)>0$ such that 
    \begin{align}
        \abso{
            \frac{1}{(2\pi)^3}
            \int_{\R^3}
            e^{i\tau \psi(y,\xi)}
            a(y,\xi)\, d\xi
            -
            \frac{e^{i\tau \psi (y,\Xi(y))}e^{\frac{\pi}{4}i \sgn \nabla^2_\xi \psi(y,\Xi(y))}}{(2\pi \tau)^{\frac{3}{2}}|\det \nabla_\xi^2 \psi (y,\Xi(y))|^{\frac{1}{2}}}
            a(y,\Xi(y))
        }
        \leq C\tau^{-\frac{5}{2}}
    \end{align}
    for all $y \in K$ and $\tau \geq 1$.
    Here, the signature of a real symmetric matrix is defined as the number of its positive eigenvalues minus the number of its negative eigenvalues.
\end{lemm}

\section{Linear analysis}\label{sec:lin}
The aim of this section is to derive the sharp decay estimates for the Stokes--Coriolis semigroup:
\begin{align}
    \mathcal{T}_\Omega(t)u_0
    ={}
    e^{t(\Delta - \Om \mathbb Pe_3 \times \mathbb{P})}u_0
    =
    e^{t\Delta}e^{it\Omega \frac{D_3}{|D|}}P_+(D) u_0
    +
    e^{t\Delta}e^{-it\Omega \frac{D_3}{|D|}}P_-(D) u_0.
\end{align}
where the Fourier multiplier $P_\pm(D)=\mathscr{F}^{-1} P_\pm(\xi)\mathscr{F}$ is defined in \eqref{def-of-PR}.
\subsection{Refined decay estimates}
To simplify the presentation, we introduce the following notation.
\begin{df}
    For $2 \leq p \leq \infty$ and $\tau \in \R$, we define
    \begin{align}
    \mathcal{D}_p(\tau)
    :=
    \begin{cases}
        {(1+|\tau|)^{-\frac{3}{2}(1-\frac{2}{p})}}  & {\rm for}\quad 2 \leq p < 4,
        \\
        {(1+|\tau|)^{-\frac{3}{4}}(\log(e+|\tau|))^{\frac{1}{4}}}  & {\rm for}\quad p=4, 
        \\
        {(1+|\tau|)^{-(1-\frac{1}{p})}}  & {\rm for}\quad 4<p \leq \infty.
    \end{cases}
   \end{align}
\end{df}
The goal of this subsection is to establish the following theorem.
\begin{thm}\label{thm:lin-decay}
    For $2 \leq p \leq \infty$, $m \in \N \cup \{0\}$,
    the following statements hold true.
    \begin{itemize}
    \item [(1)]
    For any $u_0 \in L^1(\R^3)$ with $\div u_0=0$,
    it holds
    \begin{align}
        \n{\nabla^m\mathcal{T}_\Omega(t)u_0}_{L^p}
        =
        o\sp{t^{-\frac{3}{2}(1-\frac{1}{p})-\frac{m}{2}}
            \mathcal{D}_p(|\Omega| t)}
    \end{align}
    as $t \to \infty$.
    \item [(2)]
    If $|x|u_0(x) \in L^1(\R^3)$ with $\div u_0=0$, then there exists a positive constant $C=C(p,m)$ such that
    \begin{align}
        \n{\nabla^m\mathcal{T}_\Omega(t)u_0}_{L^p}
        \leq 
        Ct^{-\frac{3}{2}(1-\frac{1}{p})-\frac{m+1}{2}}
        \mathcal{D}_p(|\Omega| t)\n{|x|u_0(x)}_{L^1(\R^3_x)}
    \end{align}
    for all $t >0$.
    \end{itemize}
\end{thm}
We first focus on the dispersive effect of the evolution group $\{e^{i\tau \frac{D_3}{|D|}}\}_{\tau \in \R}$ in $L^p$ and show the following proposition.
\begin{prop}\label{prop:dips-1}
    There exists a positive constant $C$ such that 
    \begin{align}
        \n{\Delta_j 
        { e^{i\tau \frac{D_3}{|D|}}f }  }_{L^p}
        &
        \leq 
        C
        2^{3(1-\frac{1}{p})j}
        \mathcal{D}_p(\tau)
        \n{\Delta_j f}_{L^1},
        \\
        \n{\Delta_j 
        { e^{i\tau \frac{D_3}{|D|}}f}  }_{L^{4,\infty}}
        &
        \leq 
        C
        2^{3(1-\frac{1}{4})j}
        (1+|\tau|)^{-\frac{3}{4}}
        \n{\Delta_j f}_{L^1}
    \end{align}
    for all $\tau \in \R$, $j \in \Z$ and $f \in \mathscr{S}'(\R^3)$ with $\Delta_j f \in L^1(\R^3)$.
\end{prop}
Let us focus on the estimates for the  oscillatory integral:
\begin{align}
    e^{i\tau \frac{D_3}{|D|}}\varphi(x)
    =
    \frac{1}{(2\pi)^3}
    \int_{\R^3}
    e^{i x \cdot \xi}
    e^{i \tau \Psi(\xi)}
    \widehat{\varphi}(\xi)\, d\xi,
\end{align}
where we have defined 
\begin{align}
    \widehat{\varphi}(\xi):=\widehat\varphi_{-1}(\xi) + \widehat\varphi_0(\xi) + \widehat\varphi_1(\xi),
    \qquad
    \Psi(\xi):=\frac{\xi_3}{|\xi|}.
\end{align}
Note that $\widehat \varphi$ is supported in $\{\xi \in \R^3;\ 1/4 \leq |\xi| \leq 4\}$ and it holds 
\begin{align}
    \det 
    \nabla^2 \Psi(\xi)
    =
    \frac{|\xih|^2\xi_3}{|\xi|^9}.
\end{align}
Let $\varrho \in C_c^\infty(\R ; [0,1])$ be an even function satisfying $\varrho(s) = 1$ for $|s| \leq 1/8$ and $\varrho(s) =0$ for $|s|\geq 1/4$. We define 
\begin{align}
    \rho^{\rm h}(\xi):=\varrho\sp{\frac{|\xih|}{|\xi|}},
    \qquad
    \rho^{\rm v}(\xi):=\varrho\sp{\frac{\xi_3}{|\xi|}}, 
    \qquad 
    \rho^{0}(\xi):=1 - \rho^{\rm h}(\xi) - \rho^{\rm v}(\xi)
\end{align}
for $\xi \in \R^3\setminus\{0\}$.
Observe that the supports of $\rho^{\rm h}$ and $\rho^{\rm v}$ are disjoint, and thus it holds $0 \leq \rho_0(\xi) \leq 1$.
Let us define
\begin{align}
    I^0(\tau,x)
    &:=
    \frac{1}{(2\pi)^3}
    \int_{\R^3}
    e^{i x \cdot \xi}e^{i \tau \Psi(\xi)}
    a^0(\xi)\, d\xi,
    \qquad
    a^0(\xi)
    :=\widehat \varphi (\xi) \rho^0(\xi),
    \\
    I_k^{\rm h}(\tau,x)
    &:=
    \frac{1}{(2\pi)^3}
    \int_{\R^3}
    e^{i x \cdot \xi}e^{i \tau \Psi(\xi)}
    a_k^{\rm h}(\xi)\, d\xi,
    \qquad 
    a_k^{\rm h}(\xi)
    :=\widehat \varphi (\xi) \rho^{\rm h}(\xi) \phi(2^{-k}|\xih|^2),
    \\
    I_k^{\rm v}(\tau,x)
    &:=
    \frac{1}{(2\pi)^3}
    \int_{\R^3}
    e^{i x \cdot \xi}e^{i \tau \Psi(\xi)}
    a_k^{\rm v}(\xi)\, d\xi,
    \qquad
    a_k^{\rm v}(\xi)
    :=\widehat \varphi (\xi) \rho^{\rm v}(\xi) \phi(2^{-k}|\xi_3|).
\end{align}
We then see that 
\begin{align}
    e^{i \tau \frac{D_3}{|D|}}\varphi(x)
    =
    I^0(\tau,x)
    +
    \sum_{k \leq 1}
    I_k^{\rm h}(\tau,x)
    +
    \sum_{k \leq 1}
    I_k^{\rm v}(\tau,x).
\end{align}
\begin{lemm}\label{lemm:I}
There exists a positive constant $C$ such that 
\begin{align}
    \n{I^0(\tau,\cdot)}_{L^p}
    &\leq 
    C|\tau|^{-\frac{3}{2}(1-\frac{2}{p})}, \label{est:I^0}
    \\
    \n{I_k^{\#}(\tau,\cdot)}_{L^p}
    &\leq 
    C\min\Mp{2^{(1-\frac{1}{p})k},2^{(\frac{2}{p}-\frac{1}{2})k}|\tau|^{-\frac{3}{2}(1-\frac{2}{p})}} 
\end{align}
for all $\# \in \{{\rm h},{\rm v}\}$, $k \leq 1$, $2 \leq p \leq \infty$, and $\tau \in \R$ with $|\tau| \geq 2$.
\end{lemm}
\begin{proof}
It follows from the Plancherel theorem that 
\begin{align}
    \n{I^0(\tau,\cdot)}_{L^2}
    &=
    (2\pi)^{-\frac{3}{2}}
    \n{a^0}_{L^2},
    \\
    \n{I_k^{\#}(\tau,\cdot)}_{L^2}
    &=
    (2\pi)^{-\frac{3}{2}}
    \n{a_k^{\#}}_{L^2}
    \leq 
    C2^{\frac{1}{2}k}.
\end{align}
Hence, by the interpolation, the proof is reduced to the case of $p=\infty$.

\noindent
{\it Step 1. $L^\infty$ estimates for $I^0$.}
Since it holds $|\det \nabla^2 \Psi(\xi)| = |\xih|^2|\xi_3|/|\xi|^9\geq 2^{-21}$ for $\xi \in \supp a^0$, we see by Lemma \ref{lemm:sta} that 
\begin{align}
    \n{I^0(\tau,\cdot)}_{L^\infty} \leq C|\tau|^{-\frac{3}{2}}. 
\end{align}

\noindent
{\it Step 2. $L^\infty$ estimates for $I_k^{\rm h}$.}
We see that 
\begin{align}\label{est:I^h-1}
    \n{I_k^{\rm h}(\tau,\cdot)}_{L^\infty}
    \leq 
    \frac{1}{(2\pi)^3}
    \abso{\supp a_k^{\rm h}}
    \leq 
    C2^k.
\end{align}
It follows from the change of variable $(\eta_{\rm h},\eta_3)=\sp{2^{-\frac{1}{2}k}\xih,\xi_3}$ that
\begin{align}
    I_k^{\rm h}(\tau,x)
    =
    \frac{2^k}{(2\pi)^3}
    \sum_{\sigma \in \{\pm\}}
    e^{\sigma i \tau}
    \int_{\R^3}
    e^{i (2^{\frac{1}{2}k}x_{\rm h},x_3) \cdot \eta}
    e^{-\sigma i 2^k\tau \Psi_{\sigma,k}^{\rm h}(\eta)}
    b_{\sigma,k}^{\rm h}(\eta)\, d\eta,
\end{align}
where we have defined
\begin{align}
    \Psi_{\pm,k}^{\rm h}(\eta)
    :={}&
    2^{-k} \sp{1\mp\Psi\sp{2^{\frac{1}{2}k}\eta_h,\eta_3}}
    \\
    ={}&
    |\eta_{\rm h}|^2
    \sp{2^k|\eta_{\rm h}|^2 + \eta_3^2}^{-\frac{1}{2}}
    \Mp{
    |\eta_3|+\sp{2^k|\eta_{\rm h}|^2 + \eta_3^2}^{\frac{1}{2}}
    }^{-1},
    \\
    b_{\pm,k}^{\rm h}(\eta)
    :={}&
    \widehat 
    \varphi \sp{2^{\frac{1}{2}k}\eta_{\rm h},\eta_3} 
    \varrho \sp{2^{\frac{1}{2}k}|\eta_{\rm h}|\sp{2^k|\eta_{\rm h}|^2 + \eta_3^2}^{-\frac{1}{2}}} 
    \phi(|\eta_{\rm h}|^2)
    \bm{1}_{(0,\infty)}(\pm \eta_3).
\end{align}
Elementary computations yield
\begin{align}
    &
    \bigcup_{k \leq 1}
    \supp b_{\pm, k}^{\rm h} 
    \subset 
    K^{\rm h}:=\Mp{
    \eta \in \R^3\ ;\ 
    \frac{1}{2} \leq |\eta_{\rm h}|^2 \leq 2,\ \frac{\sqrt{15}}{16} \leq |\eta_3| \leq 4
    },
    \\
    &
    \sum_{m=0}^6
    \sup_{k \leq 1}
    \n{\nabla^m\Psi_{\pm,k}^{\rm h}}_{L^\infty(K^{\rm h})}
    <\infty,
    \qquad
    \sum_{m=0}^4
    \sup_{k \leq 1}
    \n{\nabla^mb_{\pm, k}^{\rm h}}_{L^\infty(K^{\rm h})}
    <\infty,
\end{align}
and
\begin{align}
    \abso{
    \det \nabla^2 \Psi_{\pm,k}^{\rm h}(\eta)
    }
    =
    |\eta_{\rm h}|^2|\eta_3|\sp{2^k|\eta_{\rm h}|^2 + \eta_3^2}^{-\frac{9}{2}}
    \geq \sqrt{15} \cdot 2^{-23}, 
    \qquad k \leq 1,\ \eta \in K^{\rm h}.
\end{align}
From Lemma \ref{lemm:sta}, we see that
\begin{align}\label{est:I^h-2}
    \n{I_k^{\rm h}(\tau,\cdot)}_{L^\infty}
    \leq 
    C2^k\sp{2^k|\tau|}^{-\frac{3}{2}}
    =
    C2^{-\frac{1}{2}k}|\tau|^{-\frac{3}{2}}.
\end{align}
Gathering \eqref{est:I^h-1} and \eqref{est:I^h-2}, we obtain
\begin{align}
    \n{I_k^{\rm h}(\tau,\cdot)}_{L^\infty}
    \leq 
    C
    \min \Mp{2^k,2^{-\frac{1}{2}k}|\tau|^{-\frac{3}{2}}}.
\end{align}

\noindent
{\it Step 3. $L^\infty$ estimates for $I_k^{\rm v}$.}
For the estimates of $I_k^{\rm v}$,
we have 
\begin{align}\label{est:I^v-1}
    \n{I_k^{\rm v}(\tau,\cdot)}_{L^\infty}
    \leq 
    \frac{1}{(2\pi)^3}
    \abso{\supp a_k^{\rm v}}
    \leq 
    C2^k
\end{align}
It holds by the change of variable $(\eta_{\rm h},\eta_3)=(\xih,2^{-k}\xi_3)$ that
\begin{align}
    I_k^{\rm v}(\tau,x)
    =
    \frac{2^k}{(2\pi)^3}
    \int_{\R^3}
    e^{i (x_{\rm h}, 2^kx_3) \cdot \eta}
    e^{i 2^k\tau \Psi_k^{\rm v}(\eta)}
    b_k^{\rm v}(\eta)\, d\eta,
\end{align}
where we have set 
\begin{align}
    \Psi_k^{\rm v}(\eta)
    &:=
    \eta_3\sp{|\eta_{\rm h}|^2 + 2^{2k}\eta_3^2}^{-\frac{1}{2}},
    \\
    b_k^{\rm v}(\eta)
    &:=
    \widehat \varphi (\eta_{\rm h},2^k\eta_3) 
    \varrho\sp{2^k\eta_3\sp{|\eta_{\rm h}|^2 + 2^{2k}\eta_3^2}^{-\frac{1}{2}}} 
    \phi(|\eta_3|).
\end{align}
By direct calculations, 
it holds
\begin{align}
    &
    \bigcup_{k \leq 1}
    \supp b_k^{\rm v} 
    \subset 
    K^{\rm v}:=\Mp{
    \eta \in \R^3\ ;\ 
    \frac{\sqrt{15}}{16} \leq |\eta_{\rm h}| \leq 4, \ \frac{1}{2} \leq |\eta_3| \leq 2
    },
    \\
    &
    \sum_{m=0}^6
    \sup_{k \leq 1}
    \n{\nabla^m\Psi_k^{\rm v}}_{L^\infty(K^{\rm v})}
    <\infty,
    \qquad
    \sum_{m=0}^4
    \sup_{k \leq 1}
    \n{\nabla^mb_k^{\rm v}}_{L^\infty(K^{\rm v})}
    <\infty,
\end{align}
and
\begin{align}
    \abso{
    \det \nabla^2 \Psi_k^{\rm v}(\eta)
    }
    =
    |\eta_{\rm h}|^2|\eta_3|\sp{|\eta_{\rm h}|^2+2^{2k}\eta_3^2}^{-\frac{9}{2}}
    \geq 15 \cdot 2^{-27}, 
    \qquad k \leq 1,\ \eta \in K^{\rm v}.
\end{align}
Then, by Lemma \ref{lemm:sta}, we have
\begin{align}\label{est:I^v-2}
    \n{I_k^{\rm v}(\tau,\cdot)}_{L^\infty}
    \leq 
    C2^k\sp{2^k|\tau|}^{-\frac{3}{2}}
    =
    C2^{-\frac{1}{2}k}|\tau|^{-\frac{3}{2}}.
\end{align}
Combining \eqref{est:I^v-1} and \eqref{est:I^v-2}, we have 
\begin{align}
    \n{I_k^{\rm v}(\tau,\cdot)}_{L^\infty}
    \leq 
    C
    \min \Mp{2^k,2^{-\frac{1}{2}k}|\tau|^{-\frac{3}{2}}}
\end{align}
and we complete the proof.
\end{proof}
Using Lemma \ref{lemm:I}, we obtain the following $L^p$ decay estimates for $e^{i\tau \frac{D_3}{|D|}}\varphi$.
\begin{lemm}\label{lemm:dis-decay}
    For $2 \leq p \leq \infty$, 
    there exists a positive constant $C=C(p)$ such that 
    \begin{align}
        \n{e^{i\tau \frac{D_3}{|D|}}\varphi}_{L^p}
        \leq 
        C\mathcal{D}_p(\tau),
        \qquad
        \n{e^{i\tau \frac{D_3}{|D|}}\varphi}_{L^{4,\infty}}
        \leq C(1+|\tau|)^{-\frac{3}{4}}. \label{est:L4}
    \end{align}
    for all $\tau \in \R$.
\end{lemm}
\begin{proof}
    Since 
    we see for $2 \leq p \leq \infty$ that 
    \begin{align}
        \n{e^{i\tau \frac{D_3}{|D|}}\varphi}_{L^p} 
        \leq 
        \n{e^{i\tau \frac{D_3}{|D|}}\varphi}_{L^2}^{\frac{2}{p}}
        \n{e^{i\tau \frac{D_3}{|D|}}\varphi}_{L^\infty}^{1-\frac{2}{p}}
        \leq
        C
        \| \varphi \|_{L^2}^{\frac{2}{p}}
        \| \widehat \varphi \|_{L^1}^{1-\frac{2}{p}},
    \end{align}
    it suffices to consider only for sufficiently large $|\tau|$.
    Moreover, \eqref{est:I^0} implies that it is enough to focus on the estimates for 
    \begin{align}
        I^\#(\tau,x):={}&
        \sum_{k \leq 1}I_k^{\#}(\tau,x)
        ={}
        \frac{1}{(2\pi)^3}
        \int_{\R^3}
        e^{ix\cdot \xi}
        e^{i\tau \Psi(\xi)}
        a_{k}^\#(\xi)\, 
        d\xi,
        \qquad
        \# \in \{{\rm h},{\rm v}\}.
    \end{align}
    
    \noindent
    {\it Step 1. The $L^p$ estimates with $p \neq 4$.}
    It follows from Lemma \ref{lemm:I} that 
    \begin{align}
        \n{I^\#(\tau,\cdot)}_{L^p} 
        \leq{}&
        \sum_{k \leq 1}
        \n{I_k^{\#}(\tau,\cdot)}_{L^p}
        \\
        \leq{}&
        C
        \sum_{k \leq 1}\min\Mp{2^{(1-\frac{1}{p})k},2^{(\frac{2}{p}-\frac{1}{2})k}|\tau|^{-\frac{3}{2}(1-\frac{2}{p})}}
        \\
        \leq{}&
        C
        \sum_{2^k \leq |\tau|^{-1}}
        2^{(1-\frac{1}{p})k}
        +
        C|\tau|^{-\frac{3}{2}(1-\frac{2}{p})}
        \sum_{|\tau|^{-1} < 2^k \leq 2}
        2^{(\frac{2}{p}-\frac{1}{2})k}
        \\
        \leq{}&
        C\mathcal{D}_p(\tau).
    \end{align}
    
    \noindent
    {\it Step 2. The $L^{4,\infty}$ estimates.}
    Let $\delta$ be a positive constant to be determined later and define
    \begin{align}
        \mathbb{Z}_{\leq 1}
        :=
        \Mp{k \in \mathbb{Z}\ ;\ k \leq 1},
        \qquad
        E(\tau,\lambda)
        :=
        \Mp{
        k \in \mathbb{Z}_{\leq 1}\ ;\ 
        \min \Mp{2^k,2^{-\frac{1}{2}k}|\tau|^{-\frac{3}{2}}} > \delta \lambda
        }.
    \end{align}
    We then see that 
    \begin{align}
        \n{\sum_{k \in \mathbb{Z}_{\leq 1} \setminus E(\tau,\lambda)}I_k^{\#}(\tau,\cdot)}_{L^\infty}
        &
        \leq 
        \sum_{k \in \mathbb{Z}_{\leq 1} \setminus E(\tau,\lambda)}
        \n{I_k^{\#}(\tau,\cdot)}_{L^\infty}
        \\
        &\leq
        C
        \sum_{k \in \mathbb{Z}_{\leq 1} \setminus E(\tau,\lambda)}
        \min \Mp{2^k,2^{-\frac{1}{2}k}|\tau|^{-\frac{3}{2}}}
        \\
        &\leq 
        C\delta \lambda.
    \end{align}
    Since the above constant $C$ may be chosen independent of $\delta$, choosing $\delta=(2C)^{-1}$ yields 
    \begin{align}
        \n{\sum_{k \in \mathbb{Z}_{k \leq 1} \setminus E(\tau,\lambda)}I_k^{\#}(\tau,\cdot)}_{L^\infty}
        \leq  
        \frac{\lambda}{2}.
    \end{align}
    This implies 
    \begin{align}\label{include}
        \Mp{x \in \R^3\ ;\ \abso{I^{\#}(\tau,x)} > \lambda}
        \subset 
        \Mp{x \in \R^3\ ;\ \abso{\sum_{k \in E(\tau,\lambda)}I_k^{\#}(\tau,\cdot)}>\frac{\lambda}{2}}
    \end{align}
    Thus, we see by \eqref{include}, Chebyshev inequality, Lemma \ref{lemm:I} with $p=2$, and the almost orthogonality that 
    \begin{align}\label{est:dist-1}
        \abso{\Mp{x \in \R^3\ ;\ \abso{I^{\#}(\tau,x)} > \lambda}}
        \leq{}&
        \abso{\Mp{x \in \R^3\ ;\ \abso{\sum_{k \in E(\tau,\lambda)}I_k^{\#}(\tau,\cdot)}>\frac{\lambda}{2}}}
        \\
        \leq{}&
        \frac{4}{\lambda^2}
        \n{\sum_{k \in E(\tau,\lambda)}I_k^{\#}(\tau,\cdot)}_{L^2}^2
        \\
        \leq{}&
        \frac{C}{\lambda^2}
        \sum_{k \in E(\tau,\lambda)}\n{I_k^{\#}(\tau,\cdot)}_{L^2}^2
        \\
        \leq{}&
        \frac{C}{\lambda^2}
        \sum_{2^k < \delta^{-2}|\tau|^{-3}\lambda^{-2}}2^k
        \\
        \leq{}&
        C|\tau|^{-3}\lambda^{-4}.
    \end{align}
    Hence, we obtain 
    \begin{align}
        \n{I^\#(\tau,\cdot)}_{L^{4,\infty}}
        =
        \sup_{\lambda>0}
        \lambda
        \abso{\Mp{x \in \R^3\ ;\ \abso{I^{\#}(\tau,x)} > \lambda}}^{\frac{1}{4}}
        \leq 
        C|\tau|^{-\frac{3}{4}}.
    \end{align}
    
    \noindent
    {\it Step 3. The $L^{4}$ estimate.}
    The $L^\infty$ estimate for $I^\#$ yields that 
    \begin{align}\label{est:dist-0}
        \abso{\Mp{x \in \R^3\ ;\ \abso{I^\#(\tau,x)} > \lambda}}
        =0
        \qquad
        {\rm for}
        \quad
        \lambda>C|\tau|^{-1}.
    \end{align}
    By the Chebyshev inequality and the Plancherel theorem, we have 
    \begin{align}\label{est:dist-2}
        \begin{split}
        \abso{\Mp{x \in \R^3\ ;\ \abso{I^\#(\tau,x)} > \lambda}}
        &
        \leq
        \n{I^\#(\tau,\cdot)}_{L^2}^2
        \lambda^{-2}
        \\
        &=
        (2\pi)^{-3}
        \n{\widehat \varphi \rho^\#}_{L^2}^2
        \lambda^{-2}
        =C\lambda^{-2}.
        \end{split}
    \end{align}
    Hence, from \eqref{est:dist-1}, \eqref{est:dist-0}, \eqref{est:dist-2}, and the layer-cake theorem, it follows that 
    \begin{align}
        \n{I^\#(\tau,\cdot)}_{L^4}^4
        ={}&
        4
        \int_0^{C|\tau|^{-1}}
        \lambda^4
        \abso{\Mp{x \in \R^3\ ;\ \abso{I^\#(\tau,x)} > \lambda}}
        \, \frac{d\lambda}{\lambda}
        \\
        \leq{}&
        C
        \int_0^{|\tau|^{-\frac{3}{2}}}
        \lambda\, 
        d\lambda
        +
        C
        |\tau|^{-3}
        \int_{|\tau|^{-\frac{3}{2}}}^{C|\tau|^{-1}}
        \frac{d\lambda}{\lambda}
        \\
        \leq{}&
        C|\tau|^{-3}\log |\tau|.
    \end{align}
    This completes the proof.
\end{proof}
Now, we are in a position to prove Proposition \ref{prop:dips-1}.
\begin{proof}[Proof of Proposition \ref{prop:dips-1}]
    By the almost orthogonality of the Littlewood--Paley decomposition, we have 
    \begin{align}
        \Delta_j
        e^{i \tau \frac{D_3}{|D|}}
        f(x)
        =
        2^{3j}
        e^{i \tau \frac{D_3}{|D|}} \varphi (2^j \cdot)* 
        \Delta_j f.
    \end{align}
    Thus, we see by the Young inequality and Lemma \ref{lemm:dis-decay} that 
    \begin{align}
        \n{\Delta_j
        e^{i \tau \frac{D_3}{|D|}}
        f}_{L^p}
        &
        \leq 
        2^{3j}
        \n
        {e^{i \tau \frac{D_3}{|D|}} \varphi (2^j \cdot)}_{L^p} 
        \n{\Delta_j f}_{L^1}
        \\
        &=
        2^{3(1-\frac{1}{p})j}
        \n
        {e^{i \tau \frac{D_3}{|D|}} \varphi}_{L^p} 
        \n{\Delta_j f}_{L^1}
        \\
        &\leq 
        C
        2^{3(1-\frac{1}{p})j}
        \mathcal{D}_p(\tau)
        \n{\Delta_j f}_{L^1}.
    \end{align}
    For the $L^{4,\infty}$ estimate, we have by the same token that 
    \begin{align}
        \n{\Delta_j
        e^{i \tau \frac{D_3}{|D|}}
        f}_{L^{4,\infty}}
        &
        \leq 
        C
        2^{3j}
        \n
        {e^{i \tau \frac{D_3}{|D|}} \varphi (2^j \cdot)}_{L^{4,\infty}} 
        \n{\Delta_j f}_{L^1}
        \\
        &=
        C2^{3(1-\frac{1}{4})j}
        \n
        {e^{i \tau \frac{D_3}{|D|}} \varphi}_{L^{4,\infty}} 
        \n{\Delta_j f}_{L^1}
        \\
        &\leq 
        C
        2^{3(1-\frac{1}{4})j}
        (1+|\tau|)^{-\frac{3}{4}}
        \n{\Delta_j f}_{L^1}.
    \end{align}
    Thus, we complete the proof.
\end{proof}
Next, we provide the $L^p$ decay estimates for the semigroup $\Mp{e^{t\Delta}e^{i\tau \Psi(D)}}_{t>0}$.
\begin{prop}\label{prop:I-decay}
    Let $2 \leq p \leq \infty$, $m \in \N \cup \{0\}$.
    Then, there exists a positive constant $C=C(p,m)$ such that 
    \begin{align}
        \n{\nabla^m
        e^{t\Delta}e^{i\tau \Psi(D)}f}_{L^p}
        &\leq 
        Ct^{-\frac{3}{2}(1-\frac{1}{p})-\frac{m}{2}}\mathcal{D}_p(\tau)
        \n{f}_{\dB_{1,\infty}^0},
        \\
        \n{\nabla^m
        e^{t\Delta}e^{i\tau \Psi(D)}f}_{L^{4,\infty}}
        &\leq 
        Ct^{-\frac{3}{2}(1-\frac{1}{4})-\frac{m}{2}}(1+|\tau|)^{-\frac{3}{4}}
        \n{f}_{\dB_{1,\infty}^0}
    \end{align}
    for all $t>0$, $\tau \in \R$ and $f \in \dB_{1,\infty}^0(\R^3)$.
\end{prop}
\begin{proof}
    By the Littlewood--Paley decomposition and the Bernstein inequality, we deduce from Proposition \ref{prop:dips-1} that  
    \begin{align}
        \n{\nabla^m e^{t\Delta}e^{i\tau \Psi(D)}f}_{L^p}
        \leq{}&
        \sum_{j \in \Z}
        \n{ \nabla^m e^{t\Delta}  \Delta_j {e^{i\tau \Psi(D)}  f} }_{L^p}
        \\
        \leq{}& 
        C
        \sum_{j \in \Z}
        2^{mj}e^{-c2^{2j}t}
        \n{ \Delta_j  { e^{i\tau \Psi(D)}  f} }_{L^p}
        \\
        \leq{}&
        C
        \mathcal{D}_p(\tau)
        \sum_{j \in \Z}
        2^{\{m+3(1-\frac{1}{p})\}j}e^{-c2^{2j}t}
        \n{\Delta_j f}_{L^1}
        \\
        \leq{}&
        Ct^{-\frac{3}{2}(1-\frac{1}{p})-\frac{m}{2}}\mathcal{D}_p(\tau)
        \n{f}_{\dB_{1,\infty}^0}.
    \end{align}
    The $L^{4,\infty}$ estimate is obtained similarly.
    Thus, we complete the proof.
\end{proof}
\begin{proof}[Proof of Theorem \ref{thm:lin-decay}]
    Let
    \begin{align}
        K_\Om^\pm(t,x)
        :=
        \mathscr{F}^{-1}
        \lp{
        e^{\pm i\Om t \frac{\xi_3}{|\xi|}}
        e^{-t|\xi|^2}
        P_{\pm}(\xi)
        }(x).
    \end{align}
    Since $u_0 \in L^1(\R^3)$ and $\div u_0=0$, the fact shown by Miyakawa \cite{Miy-98} yields
    \begin{align}
        \int_{\R^3} u_0(y)\, dy = 0.
    \end{align}
    Thus, we have 
    \begin{align}
        \mathcal{T}_\Om(t) u_0(x)
        &
        =
        \sum_{\sigma \in \{\pm\}}
        \int_{\R^3}
        K_\Om^\sigma(t,x-y)u_0(y)\, dy
        \\
        &
        =
        \sum_{\sigma \in \{\pm\}}
        \int_{\R^3}
        \sp{
        K_\Om^\sigma(t,x-y)-K_\Om^\sigma(t,x)
        }
        u_0(y)\, dy,
    \end{align}
    which implies 
    \begin{align}\label{eq:T-L^p}
        \n{\nabla^m\mathcal{T}_\Om(t) u_0}_{L^p}
        \leq 
        \sum_{\sigma \in \{\pm\}}
        \int_{\R^3}
        \n{
        \nabla^mK_\Om^\sigma(t,\cdot -y)-\nabla^mK_\Om^\sigma(t,\cdot)
        }_{L^p}
        |u_0(y)|\, dy. 
    \end{align}
    For $R>0$, we deduce from Proposition \ref{prop:I-decay} that  
    \begin{align}
        \n{\nabla^m\mathcal{T}_\Om u_0}_{L^p}
        \leq {}&
        \sum_{\sigma \in \{\pm\}}
        \int_{|y| \leq R}
        \int_0^1\n{
        \nabla^{m+1} K_\Om^\sigma(t,\cdot -\theta y)
        }_{L^p}\,
        d\theta
        |y|
        |u_0(y)|\, dy\\
        &
        +
        2\sum_{\sigma \in \{\pm\}}\n{\nabla^mK_\Om^{\sigma}(t,\cdot)}_{L^p}\int_{|y| \geq R}|u_0(y)|\, dy
        \\
        \leq {}&
        CR
        t^{-\frac{3}{2}(1-\frac{1}{p})-\frac{m+1}{2}}\mathcal{D}_p(|\Om| t)
        \int_{\R^3}
        |u_0(y)|\, dy\\
        &
        +
        C
        t^{-\frac{3}{2}(1-\frac{1}{p})-\frac{m}{2}}\mathcal{D}_p(|\Om| t)
        \int_{|y| \geq R}|u_0(y)|\, dy,
    \end{align}
    which implies 
    \begin{align}
        \limsup_{t \to \infty}
        t^{\frac{3}{2}(1-\frac{1}{p})+\frac{m}{2}}
        \mathcal{D}_p(|\Om| t)^{-1}
        \n{\nabla^m \mathcal{T}_{\Om}(t)u_0}_{L^p}
        \leq C\int_{|y| \geq R}|u_0(y)|\, dy.
    \end{align}
    Letting $R \to \infty$, we complete the first assertion.
    For the second assertion, we have by \eqref{eq:T-L^p} that
    \begin{align}
        \n{\nabla^m \mathcal{T}_\Om u_0}_{L^p}
        \leq {}&
        \sum_{\sigma \in \{\pm\}}
        \int_{\R^3}
        \int_0^1\n{
        \nabla^{m+1} K_\Om^\sigma(t,\cdot -\theta y)
        }_{L^p}\,
        d\theta
        |y|
        |u_0(y)|\, dy\\
        \leq {}&
        Ct^{-\frac{3}{2}(1-\frac{1}{p})-\frac{m+1}{2}}
        \mathcal{D}_p(|\Om| t)
        \int_{\R^3} |y| |u_0(y)|\, dy.
    \end{align}
    This completes the proof.
\end{proof}

\subsection{Sharpness of decay rates}
In this subsection, we show that the $L^p$ decay rates of the linear solution obtained in the above section is optimal.
\begin{thm}\label{Thm:sharp-lin}
    Let $\Om=1$ and define a divergence-free vector field  $u_0 \in \mathscr{S}(\R^3)$ by
    \begin{align}
        u_0(x):=-16^{-1}\pi^{-\frac{3}{2}}e^{-\frac{|x|^2}{4}}x^{\perp},
        \qquad
        x^\perp:=\sp{x_2,-x_1,0}
    \end{align}
    for $x \in \R^3$.
    Then, for any $2 \leq p \leq \infty$,
    there exist positive constants $T_0=T_0(p)$ and $c=c(p)$ such that 
    \begin{align}
        \n{\mathcal{T}_1(t)u_0}_{L^p}
        &
        \geq ct^{-\frac{3}{2}(1-\frac{1}{p})-\frac{1}{2}}
        \mathcal{D}_p(t),
        \\
        \n{\mathcal{T}_1(t)u_0}_{L^{4,\infty}}
        &
        \geq ct^{-\frac{3}{2}(1-\frac{1}{4})-\frac{1}{2}}
        (1+t)^{-\frac{3}{4}}
    \end{align}
    for all $t \geq T_0$.
\end{thm}
\begin{proof}
Since a direct computation yields $\widehat{u_0}(\xi)=e^{-|\xi|^2}i\xi^{\perp}$,
it holds
\begin{align}
    \mathcal{T}_1(t)u_0(x)
    &=
    \frac{1}{(2\pi)^3}
    \sum_{\sigma \in \{ \pm \}}
    \int_{\R^3}
    e^{ix\cdot \xi}
    e^{\sigma i t \Psi(\xi)}
    e^{-(t+1)|\xi|^2}
    P_\sigma(\xi)i\xi^\perp\, d\xi\\
    &=
    \frac{1}{(t+1)^2}
    U\sp{ t, t^{-1}(t+1)^{-\frac{1}{2}}x },
\end{align}
where we have set 
\begin{align}
    &
    U(t,z)
    :=
    \sum_{\sigma \in \{ \pm \}}
    \frac{1}{(2\pi)^3}
    \int_{\R^3}
    e^{it \Phi^{\sigma}(y,\xi)}
    e^{-|\xi|^2}
    P_{\sigma}(\xi)i\xi^\perp\, d\xi,
    \\
    &\quad
    \Phi^\pm(y,\xi):=y\cdot \xi \pm \Psi(\xi).
\end{align}
Then, we have 
\begin{align}
    \n{\mathcal{T}_1(t)u_0}_{L^p}
    =
    \sp{t+1}^{-\frac{3}{2}(1-\frac{1}{p})-\frac{1}{2}}
    t^{\frac{3}{p}}
    \n{U(t,\cdot)}_{L^p},
\end{align}
which implies that it suffices to show
\begin{align}
    \n{U(t,\cdot)}_{L^p}
    \geq 
    \begin{cases}
        ct^{-\frac{3}{2}} & (2 \leq p < 4),\\
        ct^{-\frac{3}{2}}(\log t)^{\frac{1}{4}} & (p=4),\\
        ct^{-1-\frac{2}{p}} & (4 < p \leq \infty),
    \end{cases}
    \qquad
    \n{U}_{L^{4,\infty}} \geq ct^{-\frac{3}{2}}.
\end{align}

\noindent
{\it Step 1. $L^p$ estimates with $2 \leq p <4$ and $L^{4,\infty}$ estimate.}
Let 
\begin{align}
    K:=
    \Mp{
    y \in \R^3
    \ ;\ 
    \frac{1}{2} \leq | \yh | \leq 1,
    \quad
    -1 \leq y_3 \leq -\frac{1}{2}
    }.
\end{align}
Fix a $y \in K$.
Then, the stationary points of $\Phi^\pm(y,\cdot)$ are given by 
\begin{align}
    \Mp{\xi \in \R^3\ ;\ \nabla_\xi \Phi^\sigma (y,\xi) = 0}
    =
    \begin{cases}
        \Mp{\Xi^+(y),\Xi^-(y)} & (\sigma = +), \\
        \varnothing & (\sigma = -),
    \end{cases}
\end{align}
where we have put
\begin{align}
\Xi^\pm(y)
:=
\sp{\pm \frac{y_1y_3^2}{|\yh||y|^3},\pm \frac{y_2y_3^2}{|\yh||y|^3},\mp \frac{|\yh|y_3}{|y|^3}}.
\end{align}
Let $\chi \in C_c^{\infty}(\R^3;[0,1])$ and $\chi(0)=1$ and $\supp \chi \subset \{\xi \in \R^3\ ;\ |\xi| \leq 1\}$.
We define
\begin{align}
    \chi^\pm(\xi,y):=
    \chi \sp{1000(\xi-\Xi^\pm(y))},
    \qquad
    \chi^{{\rm r}}(\xi,y):=1-\chi^+(\xi,y)-\chi^-(\xi,y).
\end{align}
Along this decomposition, we decompose $U^+$ as
\begin{align}
    &
    U(t,y)
    =
    U^{+}(t,y)
    +
    U^-(t,y)
    +
    U^{{\rm r}}(t,y),
    \\
    &\quad
    \begin{aligned}
    U^{\pm}(t,y)
    :={}&
    \frac{1}{(2\pi)^3}
    \int_{\R^3}
    e^{it \Phi^{+}(y,\xi)}
    e^{-|\xi|^2}
    P_{+}(\xi)i\xi^\perp
    \chi^\pm(\xi,y)
    \, d\xi,
    \\
    U^{{\rm r}}(t,y)
    :={}&
    \frac{1}{(2\pi)^3}
    \int_{\R^3}
    e^{it \Phi^{+}(y,\xi)}
    e^{-|\xi|^2}
    P_{+}(\xi)i\xi^\perp
    \chi^{{\rm r}}(\xi,y)
    \, d\xi
    \\
    &
    +
    \frac{1}{(2\pi)^3}
    \int_{\R^3}
    e^{it \Phi^-(y,\xi)}
    e^{-|\xi|^2}
    P_{-}(\xi)i\xi^\perp
    \, d\xi.
    \end{aligned}
\end{align}
Lemmas \ref{lemm:asy} yields
\begin{align}
   \n{U^{\pm}(t,\cdot)-V^{\pm}(t,\cdot)}_{L^\infty(K)}
   \leq 
   Ct^{-\frac{5}{2}}.
\end{align}
for all $\sigma \in \{ \pm \}$ and $t \geq 1$, where
\begin{align}
    V^\pm(t,y)
    :=
    \frac
    {
    e^{it \Phi^+(y,\Xi^\pm(y))}
    e^{i\frac{\pi}{4} \sgn \nabla^2 \Psi(\Xi^\pm(y))}
    }
    {(2\pi t)^{\frac{3}{2}}|\det \nabla^2 \Psi(\Xi^\pm(y))|^\frac{1}{2}}
    e^{-|\Xi^\pm(y)|^2}
    P_{+}(\Xi^\pm(y))i\Xi^\pm(y)^\perp.
\end{align}
Since direct calculations show that 
\begin{align}
   |\det \nabla^2 \Psi(\Xi^\pm(y))|
   =
   \frac{|\yh||y|^9}{|y_3|^4}, 
   \qquad
   \sgn \nabla^2 \Psi(\Xi^\pm(y)) = \mp 1
\end{align}
and
\begin{align}
    |\Xi^\pm(y)|^2=\frac{y_3^2}{|y|^4},
    \qquad
    P_{+}(\Xi^\pm(y))i\Xi^\pm(y)^\perp
    =
    \frac{y_3^2}{2|y|^3}
    \sp{
    \pm 
    i
    \frac{y^\perp}{|\yh|}
    -
    \frac{y}{|y|}
    },
\end{align}
we may write $V(t,y)=V^+(t,y)+V^-(t,y)$ explicitly as 
\begin{align}
   V(t,y)
   ={}&
   \frac{1}{(2\pi t)^{\frac{3}{2}}}
   \frac{y_3^4}{|\yh|^{\frac{1}{2}}|y|^{\frac{15}{2}}}
   \exp\sp{-\frac{y_3^2}{|y|^4}}
   \sp{
   \sin \Theta(t,y)
   \frac{y^\perp}{|\yh|}
   +
   \cos \Theta(t,y)
   \frac{y}{|y|}
   },
\end{align}
where
\begin{align}
   \Theta(t,y):=\frac{|\yh|}{|y|}t-\frac{\pi}{4}.
\end{align}
Then, we see that 
\begin{align}
   |V(t,y)|
   =
   \frac{1}{(2\pi t)^{\frac{3}{2}}}
   \frac{y_3^4}{|\yh|^{\frac{1}{2}}|y|^{\frac{15}{2}}}
   \exp\sp{-\frac{y_3^2}{|y|^4}},
\end{align}
which provides
\begin{align}
   \n{V(t,\cdot)}_{\mathcal{L}^p(K)}
   =
   c_pt^{-\frac{3}{2}},
\end{align}
where $\mathcal{L}^p$ stands for $L^p$ for $2 \leq p <4$ and $L^{4,\infty}$ for $p=4$.
For the remainder term, since the phase function $\Phi^+(y,\cdot)$ has no stationary points on the support of $\chi^{{\rm r}}(\cdot,y)$ and $\Phi^-(y,\cdot)$ does not possess stationary points for $y \in K$, performing integration by parts three times yields
\begin{align}
    \n{U^{{\rm r}}(t,\cdot)}_{L^\infty(K)}
   \leq
   Ct^{-3}.
\end{align}
Hence, we obtain 
\begin{align}
   \n{U(t,\cdot)}_{\mathcal{L}^p}
   &\geq
   \n{U(t,\cdot)}_{\mathcal{L}^p(K)}
   \\
   &
   \begin{aligned}
    \geq 
    \n{V(t,\cdot)}_{\mathcal{L}^p(K)}
    &-
    \sum_{\sigma \in \{\pm\}}
    \n{U^{\sigma}(t,\cdot)-V^\sigma(t,\cdot)}_{L^\infty(K)}|K|^{\frac{1}{p}}
    \\
    &
    -
    \n{U^{{\rm r}}(t,\cdot)}_{L^\infty(K)}|K|^{\frac{1}{p}}
   \end{aligned}
   \\
   &\geq
   c_pt^{-\frac{3}{2}} - Ct^{-\frac{5}{2}}-Ct^{-3}
   \\
   &\geq
   \frac{c_p}{2}t^{-\frac{3}{2}}
\end{align}
for sufficiently large $t$.

\noindent
{\it Step 2. The $L^p$ estimates with $4<p \leq \infty$.}
Let us consider the vertical component
\begin{align}
    U_3(t,y)
    =
    \frac{1}{(2\pi)^3}
    \sum_{\sigma \in \{\pm\}}
    \frac{\sigma}{2}
    \int_{\R^3} 
    e^{i t \Phi^\sigma(y,\xi)}e^{-|\xi|^2}\frac{|\xih|^2}{|\xi|}\, d\xi
\end{align}
for $t \geq 1$ and $y \in \R^3$ with $1 \leq y_3 \leq 2$.
Using the $3$D polar coordinate $\xi=(r\sin \eta \cos \theta,r \sin \eta \sin \theta,r\cos \eta)$ and suitable rotation of the horizontal coordinate, we see that 
\begin{align}
    U_3(t,y)
    &={}
    \frac{1}{(2\pi)^3}
    \sum_{\sigma \in \{\pm\}}
    \frac{\sigma}{2}
    \int_0^\infty
    \int_0^\pi 
    \int_{-\pi}^{\pi}
    e^{it \phi^\sigma (y,r,\eta,\theta)}
    e^{-r^2}r^3\sin^3\eta 
    dr d\eta d \theta,
    \\ 
    &={}
    \frac{1}{(2\pi)^2}
    \sum_{\sigma \in \{\pm\}}
    \frac{\sigma}{2}
    \int_0^\infty
    r^3e^{-r^2}
    \int_0^\pi 
    J(tr|\yh|\sin \eta)
    e^{it(ry_3 \sigma 1)\cos \eta}
    \sin^3\eta \,
    dr d\eta\\
    &
    =
    \frac{1}{2(2\pi)^2}\sp{I^+(t,y)-I^-(t,y)},
\end{align}
where we have put
\begin{align}
    \phi^\pm(y,r,\eta,\theta)
    &:=
    r|\yh|\sin \eta \cos \theta+(ry_3 \pm 1)\cos \eta,
    \\
    J(\rho)
    &:=\frac{1}{2\pi} \int_{-\pi}^\pi e^{i\rho\cos \theta}\, d\theta
\end{align}
and 
\begin{align}
    I^\pm(t,y)
    :={}&
    \int_0^\infty
    r^3e^{-r^2}H(tr|\yh|,t(ry_3\pm 1))
    \, dr,
    \\
    H(\rho,\tau)
    :={}&
    \int_{0}^\pi
    J(\rho\sin \eta)
    e^{i\tau\cos \eta}
    \sin^3\eta \, d\eta\\
    ={}&
    \int_{-1}^1
    h(\rho,\zeta)
    e^{i\tau\zeta}
     \, d\zeta,\\
     h(\rho,\zeta)
     :={}&
     (1-\zeta^2)J\sp{\rho\sqrt{1-\zeta^2}}.
\end{align}

Let us investigate the properties of $H$.
By the integration by parts, we have 
\begin{align}
    H(\rho,\tau)
    ={}&
    -\frac{1}{i\tau}
    \int_{-1}^1
    \partial_\zeta h(\rho,\zeta) e^{i\tau\zeta}\, 
    d\zeta\\
    ={}&
    \frac{1}{\tau^2}
    \Big[\partial_\zeta h(\rho,\zeta) e^{i\tau\zeta}\Big]_{\zeta=-1}^{\zeta=1}
    -
    \frac{1}{\tau^2}
    \int_{-1}^1
    \partial^2_\zeta h(\rho,\zeta) e^{i\tau\zeta}\, 
    d\zeta\\
    ={}&
    \frac{4\cos \tau}{\tau^2}
    +
    \frac{2}{\tau^2}
    \int_{-1}^1
    J(\rho\sqrt{1-\zeta^2})
    e^{i\tau\zeta}\, 
    d\zeta
    \\
    &
    -
    \frac{\rho}{\tau^2}
    \int_{-1}^1
    \frac{4\zeta^2-1}{\sqrt{1-\zeta^2}}J'(\rho\sqrt{1-\zeta^2})
    e^{i\tau\zeta}\, 
    d\zeta
    \\
    &
    -
    \frac{\rho^2}{\tau^2}
    \int_{-1}^1
    \zeta^2J''(\rho\sqrt{1-\zeta^2})
    e^{i\tau\zeta}\, 
    d\zeta.
\end{align}
From this, we see that 
\begin{align}
    &
    \partial_\rho H(\rho,\tau)
    ={}
    -
    \frac{2\rho}{\tau^2}
    \int_{-1}^1
    \frac{\zeta}{\sqrt{1-\zeta^2}}
    J'(\rho\sqrt{1-\zeta^2})
    e^{i\tau\zeta}\, 
    d\zeta
    \\
    &\quad 
    -
    \frac{1}{\tau^2}
    \int_{-1}^1
    \frac{4\zeta^2-1}{\sqrt{1-\zeta^2}}J'(\rho\sqrt{1-\zeta^2})
    e^{i\tau\zeta}\, 
    d\zeta
    -
    \frac{\rho^2}{\tau^2}
    \int_{-1}^1
    \sp{4\zeta^2-1}J'(\rho\sqrt{1-\zeta^2})
    e^{i\tau\zeta}\, 
    d\zeta
    \\
    &
    \quad 
    -
    \frac{2\rho}{\tau^2}
    \int_{-1}^1
    \zeta^2J''(\rho\sqrt{1-\zeta^2})
    e^{i\tau\zeta}\, 
    d\zeta
    -
    \frac{\rho^2}{\tau^2}
    \int_{-1}^1
    \zeta^2\sqrt{1-\zeta^2}
    J'''(\rho\sqrt{1-\zeta^2})
    e^{i\tau\zeta}\, 
    d\zeta.
\end{align}
Then, using $|J(s)|+|J'(s)|+|J''(s)|+|J'''(s)| \lesssim 1$ and $|J'(s)| \lesssim |s|$,
we have
\begin{align}\label{est:H}
    |H(\rho,\tau)| \leq C \frac{1+\rho^2}{1+\tau^2},
    \qquad
    |\partial_\rho H(\rho,\tau)| \leq C \frac{|\rho|+\rho^2}{1+\tau^2}.
\end{align}
Moreover,
since
$H(\rho,\tau)=\mathscr{F}\lp{h(\rho,\cdot)\bm{1}_{[-1,1]}}(\tau)$,
it holds
\begin{align}
    \int_{\R}
    H(\rho,\tau)\, d\tau 
    =
    2\pi h(\rho,0)
    =
    2\pi J(\rho).
\end{align}

For the estimate of $I^-(\tau,y)$, the change of the variable $\tau=t(ry_3-1)$ yields
\begin{align}
    I^-(t,y)
    =
    \frac{1}{ty_3}
    \int_{-t}^\infty
    g\sp{\frac{\tau+t}{ty_3}}
    H\sp{\frac{(\tau+t)|\yh|}{y_3},\tau}\, 
    d\tau,
\end{align}
where $g(r):=r^3e^{-r^2}$.
Then, we see that
\begin{align}
    &
    I^-(t,y)
    -
    \frac{2\pi}{ty_3^4}
    \exp\sp{-\frac{1}{y_3^2}}
    J\sp{\frac{t|\yh|}{y_3}}
    \\
    &\quad
    =
    I^-(t,y)
    -
    \frac{1}{ty_3}
    g\sp{\frac{1}{y_3}}
    \int_{\R}
    H\sp{\frac{t|\yh|}{y_3},\tau}\, d\tau 
    \\
    &
    \quad
    =E_1(t,y)+E_2(t,y)+E_3(t,y)+E_4(t,y),
\end{align}
where 
\begin{align}
    E_1(t,y)
    &:=
    \frac{1}{ty_3}
    g\sp{\frac{1}{y_3}}
    \int_{-\infty}^{-t}
    H\sp{\frac{t|\yh|}{y_3},\tau}\,
    d\tau,
    \\
    E_2(t,y)
    &:=
    \frac{1}{ty_3}
    \int_{-t}^t
    \sp{
    g\sp{\frac{\tau+t}{ty_3}}
    -
    g\sp{\frac{1}{y_3}}
    }
    {H\sp{\frac{t|\yh|}{y_3},\tau}}
    \, d\tau,
    \\
    E_3(t,y)
    &:=
    \frac{1}{ty_3}
    \int_{-t}^t
    g\sp{\frac{\tau+t}{ty_3}}
    \sp{
    H\sp{\frac{(\tau+t)|\yh|}{y_3},\tau}
    -
    H\sp{\frac{t|\yh|}{y_3},\tau}}
    \, d\tau
    \\
    E_4(t,y)
    &:=
    \frac{1}{ty_3}
    \int_{t}^{\infty}
    g\sp{\frac{\tau+t}{ty_3}}
    H\sp{\frac{(\tau+t)|\yh|}{y_3},\tau}
    -
    g\sp{\frac{1}{y_3}}
    H\sp{\frac{t|\yh|}{y_3},\tau}\, d\tau,
\end{align}
Making use of \eqref{est:H}, we have
\begin{align}
    &
    \begin{aligned}
    \sum_{m=1,4}
    \sup_{1 \leq y_3 \leq 2}
    |E_m(t,y)|
    \leq {}&
    C
    \frac{1+t^2|\yh|^2}{t}
    \int_{|\tau|\geq t}
    \frac{1}{1+\tau^2}\, 
    d\tau \\    
    \leq {}&
    C\frac{1+t^2|\yh|^2}{t^2}
    \end{aligned}
    \\
    &
    \begin{aligned}
    \sup_{1 \leq y_3 \leq 2}
    |E_2(t,y)|
    \leq {}&
    \int_{-t}^t
    \frac{|\tau|}{t}
    \abso{H\sp{\frac{t|\yh|}{y_3},\tau}}
    \, d\tau
    \\
    \leq{}&
    C
    \frac{1+t^2|\yh|^2}{t^2}
    \int_{-t}^{t}
    \frac{|\tau|}{1+\tau^2}\, 
    d\tau\\
    \leq{}&
    C(1+t^2|\yh|^2)
    \frac{\log t}{t^2},
    \end{aligned}
    \\
    &
    \begin{aligned}
    \sup_{1 \leq y_3 \leq 2}
    |E_3(t,y)|
    \leq {}&
    \int_{-t}^t
    \frac{|\tau||\yh|}{t}
    \sup_{0 \leq \rho \leq 2t|\yh|/y_3}
    \abso{\partial_\rho H\sp{\rho,\tau}}
    \, d\tau
    \\
    \leq{}&
    C
    \int_{-t}^{t}
    \frac{|\tau||\yh|}{t}
    \frac{t|\yh|(1+t|\yh|)}{1+\tau^2}\, 
    d\tau\\
    \leq{}&
    C
    |\yh|^2(1+t|\yh|)
    \log t.
    \end{aligned}
\end{align}
Thus, we obtain
\begin{align}
    &
    \sup_{1\leq y_3 \leq 2}
    \abso{
    I^-(t,y)
    -
    \frac{2\pi}{ty_3^4}
    \exp\sp{-\frac{1}{y_3^2}}
    J\sp{\frac{t|\yh|}{y_3}}
    }
    \\
    &\quad 
    \leq 
    C(1+t^2|\yh|^2)\frac{\log t}{t^2} 
    + 
    C|\yh|^2(1+t|\yh|)\log t.
\end{align}
For the estimate of $I^+(t,y)$, we have by \eqref{est:H} that 
\begin{align}
    |I^+(t,y)|
    \leq {}&
    \int_0^\infty
    r^3e^{-r^2}|H(tr|\yh|,t(ry_3+ 1))|
    \, dr\\
    \leq {}&
    C
    \int_0^\infty
    r^3e^{-r^2}
    \frac{1+t^2r^2|\yh|^2}{1+t^2(ry_3+1)^2}
    \, dr\\
    \leq {}&
    \frac{C}{t^2}
    \int_0^\infty
    r^3e^{-r^2}
    \sp{1+t^2r^2|\yh|^2}
    \, dr
    =
    C\frac{1+t^2|\yh|^2}{t^2}.
\end{align}
Let $K_{t,\delta}:=\Mp{y \in \R^3\ ;\ |\yh| \leq \delta t^{-1},\ 1 \leq y_3 \leq 2}$ for $0<\delta \leq 1$.
Then,  
we have 
\begin{align}
    \n{U(t,\cdot)}_{L^p}
    \geq{}&
    \n{U_3(t,\cdot)}_{L^p(K_{t,\delta})}
    \\
    \geq{}& 
    c\n{
    \frac{2\pi}{ty_3^4}
    \exp\sp{-\frac{1}{y_3^2}}
    J\sp{\frac{t|\yh|}{y_3}}}_{L^p(K_{t,\delta})}\\
    &
    -
    C
    \n{
    I^-(t,y)
    -
    \frac{2\pi}{ty_3^4}
    \exp\sp{-\frac{1}{y_3^2}}
    J\sp{\frac{t|\yh|}{y_3}}
    }_{L^\infty(K_{t,\delta})}|K_{t,\delta}|^{\frac{1}{p}}
    \\
    &
    -
    C\n{I^+(t,\cdot)}_{L^\infty(K_{t,\delta})}|K_{t,\delta}|^{\frac{1}{p}}
    \\
    \geq{}&
    c\sp{\inf_{|r|\leq C\delta}|J(r)|
    -C\frac{\log t}{t}-\frac{C}{t}}\frac{|K_{t,\delta}|^{\frac{1}{p}}}{t}
    \\
    \geq{}&
    ct^{-1-\frac{2}{p}}
\end{align}
for $t \geq e$ and sufficiently small $\delta$.

\noindent
{\it Step 3. The $L^4$ estimate.}
Let $K_t:=\{y \in \R^3\ ;\ |\yh| \leq t^{-\frac{7}{8}},\ 1 \leq y_3 \leq 2\}$.
By the similar calculation as the last estimate in the above step, we have 
\begin{align}
    \n{U(t,\cdot)}_{L^4}
    \geq{}&
    \n{U_3(t,\cdot)}_{L^4(K_t)}
    \\
    \geq{}& 
    c\n{
    \frac{2\pi}{ty_3^4}
    \exp\sp{-\frac{1}{y_3^2}}
    J\sp{\frac{t|\yh|}{y_3}}}_{L^4(K_t)}\\
    &
    -
    C
    \n{
    I^-(t,y)
    -
    \frac{2\pi}{ty_3^4}
    \exp\sp{-\frac{1}{y_3^2}}
    J\sp{\frac{t|\yh|}{y_3}}
    }_{L^\infty(K_{t})}|K_t|^{\frac{1}{4}}
    \\
    &
    -
    C\n{I^+(t,\cdot)}_{L^\infty(K_t)}|K_t|^{\frac{1}{4}}
    \\
    \geq{}&
    \frac{c}{t}\n{J\sp{\frac{t|\yh|}{y_3}}}_{L^4(K_t)}
    -C\frac{\log t}{t^{\frac{13}{8}}}|K_t|^{\frac{1}{4}}
    -\frac{C}{t^{\frac{7}{4}}}|K_t|^{\frac{1}{4}}
    \\
    \geq{}&
    \frac{c}{t}\n{J\sp{\frac{t|\yh|}{y_3}}}_{L^4(K_t)}
    -C\frac{\log t}{t^{\frac{33}{16}}}.
\end{align}
Changing the horizontal variable to the polar coordinate, we estimate the leading term as 
\begin{align}
    \n{J\sp{\frac{t|\yh|}{y_3}}}_{L^4(K_t)}^4
    &=
    2\pi
    \int_1^2
    \int_0^{t^{-7/8}}
    \abso{J\sp{\frac{tr}{y_3}}}^4r\, drdy_3\\
    &=
    \frac{2\pi}{t^2}
    \int_1^2
    y_3^2
    \int_0^{t^{1/8}/y_3}
    \abso{J\sp{\rho}}^4\rho\, d\rho dy_3\\
    &\geq 
    \frac{c}{t^2}
    \int_1^{ct^{1/8}}
    \abso{J\sp{\rho}}^4\rho\, d\rho.
\end{align}
Since $J$ is the zeroth order Bessel function, its asymptotic expansion is given by 
\begin{align}
    \abso{J(\rho) - \sp{\frac{2}{\pi \rho}}^{\frac{1}{2}}\cos\sp{\rho-\frac{\pi}{4}}} \leq C\rho^{-\frac{3}{2}}
\end{align}
for $\rho \geq 1$; see \cite{Gra-14}*{Section B.8}.
Using this and $\cos^4 \theta = (3+4\cos 2\theta + \cos 4 \theta )/8$, we have
\begin{align}
    \int_1^{ct^{1/8}}
    \abso{J\sp{\rho}}^4\rho\, d\rho
    \geq{}& 
    c
    \int_1^{ct^{1/8}}
    \cos^4\sp{\rho-\frac{\pi}{4}}
    \frac{d\rho}{\rho}
    -C\int_1^{\infty}\rho^{-5}\, d\rho 
    \\
    \geq{}&
    c
    \int_1^{ct^{1/8}}
    \frac{d\rho}{\rho}
    \\
    &-
    C
    \abso{\int_1^{ct^{1/8}} \frac{\sin (2\rho)}{\rho}\, d\rho}
    -
    C
    \abso{\int_1^{ct^{1/8}} \frac{\cos (4\rho)}{\rho}\, d\rho}
    -
    C
    \\
    \geq{}&
    c\log t
\end{align}
for sufficiently large $t$.
Combining above estimates, we obtain
\begin{align}
    \n{U(t,\cdot)}_{L^4}
    \geq
    Ct^{-\frac{3}{2}}(\log t)^{\frac{1}{4}},
\end{align}
which completes the proof.
\end{proof}
    \begin{rem}
        Improving the calculations in the above proof with $2 \leq p <4$, one may prove the following asymptotic expansion of the linear solution.
        \begin{align}
            \lim_{t \to \infty}
            t^{\frac{3}{2}(1-\frac{1}{p})+\frac{1}{2}}(1+|\Om|t)^{\frac{3}{2}(1-\frac{2}{p})}
            \n{\mathcal{T}_\Om(t)u_0 - \mathcal{U}_\Om [u_0](t,\cdot)}_{L^p}
            =0
            \quad
            {\rm for}
            \quad
            2 \leq p < 4,
        \end{align}
        where 
        \begin{align}
            \mathcal{U}_\Om[u_0](t,x)
            :={}&
			\frac{(|\Omega|t)^{\frac{5}{2}}}{(2\pi)^{\frac{3}{2}}}
			\frac{x_3^2}{|\xh|^{\frac{1}{2}}|x|^{\frac{9}{2}}}
			\exp\sp{-\Omega^2t^3\frac{x_3^2}{|x|^4}}
			\left\{ \sin\sp{|\Omega| t\frac{|\xh|}{|x|}-\frac{\pi}{4}}I\right. \\
			&
			\left. +\sgn(\Omega)\cos \sp{|\Omega| t\frac{|\xh|}{|x|}-\frac{\pi}{4}}
			Q(x)\right\}
			M[u_0],
        \end{align}
        where
        \begin{align}
            Q(x)
			&:=
			\frac{1}{|\xh||x|}
			\begin{pmatrix}
				0 & |\xh|^2 & x_2x_3 \\
				-|\xh|^2 & 0 & -x_1x_3 \\
				-x_2x_3 & x_1x_3 & 0
			\end{pmatrix},
			\\
			M[u_0](x)
			&:=
			\sum_{k=1}^2
			\frac{x_k x_3^2}{|\xh||x|^3}
			\int_{\R^3}y_k u_0(y)\, dy
			-
			\frac{|\xh|x_3}{|x|^3}
			\int_{\R^3}y_3 u_0(y)\, dy,
        \end{align}
        provided that $|x| u_0(x) \in L^1(\R^3)$.
        We do not present the precise proof of this expansion since the asymptotic analysis is not our main purpose of this paper.
    \end{rem}

\section{Nonlinear analysis}\label{sec:non}
In this section, we present the proofs of our main results.
We first recall the global well-posedness.
\begin{lemm}\label{lemm:GWP}
    For any $u_0 \in H^{\frac{1}{2}}(\R^3)$ with $\div u_0 = 0$,
    there exists a positive constant $\Omega_0=\Omega_0(u_0)$ suitably large such that if $|\Om| \geq \Om_0$,
    the system \eqref{eq:nonlin} possesses a unique global solution $u$ satisfying
    \begin{align}
        u \in C([0,\infty);H^{\frac{1}{2}}(\R^3)),
        \quad
        \nabla u \in L^2(0,\infty;H^{\frac{1}{2}}(\R^3))
    \end{align}
    and, for every $m \in \N$, it holds $\nabla^m u \in C((0,\infty);H^{\frac{1}{2}}(\R^3))$ with the estimate
    \begin{align}
        \sup_{t>0}
        t^{\frac{m}{2}}\n{\nabla^m u(t)}_{H^{\frac{1}{2}}} < \infty.
    \end{align}
    Moreover if we additionally assume $u_0 \in L^1(\R^3)$,
    then there exists a positive constant $K_{0,m}=K_{0,m}(\n{u_0}_{L^1},\n{u_0}_{\dH^{\frac{1}{2}}})$ such that 
    \begin{align}
        \n{\nabla^mu(t)}_{L^2} \leq K_{0,m}(1+t)^{-\frac{3}{4}}t^{-\frac{m}{2}}
    \end{align}
    for all $t > 0$.
    Furthermore, for $0 < \varepsilon \leq 1$, it holds $K_{0,m}=O_m(\varepsilon)$ if $\n{u_0}_{L^1 \cap \dH^{\frac{1}{2}}}\leq \varepsilon$.
\end{lemm}
We omit the proof of this lemma since it is obtained by the global well-posedness theory of \cite{Iwa-Tak-13} 
and the Fourier splitting method established in \cite{Sch-Wie-96}.

Now we focus on the decay estimates of the nonlinear term given by
\begin{align}
\mathcal{N}_{\Omega}[u](t)
        :=
        \int_0^{t} \mathcal{T}_\Omega(t-\tau)\mathbb{P}\div (u \otimes u)(\tau)\, d\tau . 
\end{align}

\begin{prop}\label{prop:nonl-dec}
    Let $u_0 \in \dot H^{\frac{1}{2}}(\R^3) \cap L^1(\R^3)$ with $\div u_0 = 0$,
    and let $u$ be the associated global solution to \eqref{eq:nonlin} with $|\Omega| \geq \Omega_0$, constructed in Lemma \ref{lemm:GWP}.
    Then, for every $2 \leq p < \infty$ and $ m \in \N \cup \{0\}$, there exists a positive constant $K_{p,m}=K_{p,m}(\n{u_0}_{L^1},\n{u_0}_{\dH^{\frac{1}{2}}})$, independent of $\Om$, such that 
    \begin{align}
        \n{\nabla^m\mathcal{N}_\Om[u](t)}_{L^p}
        &
        \leq 
        K_{p,m}
        t^{-\frac{3}{2}(1-\frac{1}{p})-\frac{m+1}{2}}
        \mathcal{D}_p(|\Om| t), \label{Prop:dec-p}
        \\
        \n{\nabla^m\mathcal{N}_\Om[u](t)}_{L^{4,\infty}}
        &
        \leq 
        K_{4,m}
        t^{-\frac{3}{2}(1-\frac{1}{4})-\frac{m+1}{2}}
        (1+|\Om|t)^{-\frac{3}{4}} \label{Prop:dec-p-44}
    \end{align}
    for all $t>0$.
    Moreover, for the case of $p=\infty$, there exist a positive constant $K_{\infty,m}=K_{\infty,m}(\n{u_0}_{L^1},\n{u_0}_{\dH^{\frac{1}{2}}})$ and an absolute positive constant $C$ such that
    \begin{align}
        \n{\nabla^m\mathcal{N}_\Om[u](t)}_{L^\infty}
        \leq {}&
        K_{\infty,m}
        t^{-\frac{3}{2}-\frac{m+1}{2}}
        \mathcal{D}_\infty(|\Om| t) 
        \\
        &
        +
        Ct^{-\frac{5}{2}-\frac{m}{2}}
        \mathcal{D}_4(|\Om| t)^2
        \n{u}_{X_{\Om,4}^{0}(t)}
        \n{u}_{X_{\Om,4}^{m+1}(t)}^{\frac{1}{2}}
        \n{u}_{X_{\Om,4}^{m+2}(t)}^{\frac{1}{2}}
    \end{align}
    for all $t>0$, where 
    \begin{align}
        \n{u}_{X_{\Om,p}^{m}(t)}
        :=
        \sup_{\frac{t}{2} \leq \tau \leq t}
        \tau^{\frac{3}{2}(1-\frac{1}{p})+\frac{m}{2}}
        \mathcal{D}_p(|\Om|\tau)^{-1}\n{\nabla^m u(\tau)}_{L^p}.
    \end{align}
    Furthermore, it holds $K_{p,m} = O_{p,m}(\varepsilon^2)$ for $2 \leq p \leq \infty$, provided that  $\n{u_0}_{L^1 \cap \dH^{\frac{1}{2}}} \leq \varepsilon$ with $0<\varepsilon \leq 1$.
\end{prop}
To show this proposition, we use the following elementary estimate.
\begin{lemm}\label{lemm:elem}
    For $2 \leq p \leq \infty$, there exists a positive constant $C=C(p)$ such that 
    \begin{align}
        \int_0^t \mathcal{D}_p(|\Om|s)\, ds \leq 
        \begin{cases}
            Ct\mathcal{D}_p(|\Om| t) & (2 \leq p<\infty), \\
            C|\Om|^{-1}\log (e + |\Om| t) & (p=\infty)
        \end{cases}
    \end{align}
    for $t > 0$ and $\Om \in \R \setminus \{0\}$.
\end{lemm}
\begin{proof}
The case of $p=2$ is trivial. For the case of $p=\infty$, we see that 
\begin{align}
 \int_0^t \sp{1+ |\Om| \tau }^{-1} \, d\tau 
 & = |\Om|^{-1} \int_0^{|\Om| t} \sp{1+s}^{-1} \,ds
 \\
 &
 =|\Om|^{-1} \log \sp{ 1+ |\Om| t } \leq |\Om|^{-1} \log \sp{ e+ |\Om| t }.
\end{align}
Next, we focus on the case $2<p<\infty$. In case of $p \neq 4$, we use the estimate
\begin{align}
    \int_0^t (1+s)^{-\gamma}\, ds = \frac{(1+t)^{1-\gamma}-1}{1-\gamma}
    \leq \frac{t(1+t)^{-\gamma}}{1-\gamma}
\end{align}
for $0<\gamma<1$ and $t>0$.
Let 
\begin{align}
    \gamma_p:=
    \begin{cases}
        \frac{3}{2}(1-\frac{2}{p}) & (2 < p < 4), \\
        1 - \frac{1}{p} & (4 < p < \infty).
    \end{cases}
\end{align}
Then, we have for $p \neq 4$ that
\begin{align}
    \int_0^t \mathcal{D}_p(|\Om| \tau)\, d\tau 
    &=
    |\Om|^{-1}
    \int_0^{|\Om|t}
    (1+s)^{-\gamma_p}\, 
    ds
    \\
    &\leq \frac{t(1+|\Om|t)^{-\gamma_p}}{1-\gamma_p}
    =
    \frac{1}{1-\gamma_p}t\mathcal{D}_p(|\Om| t).
\end{align}
For the case of $p=4$, we have 
\begin{align}
    \int_0^t \mathcal{D}_4(|\Om| \tau)\, d\tau 
    &= 
    |\Om|^{-1}
    \int_0^{|\Om| t}
    (1+s)^{-\frac{3}{4}}(\log (e+s))^{\frac{1}{4}}\, ds
    \\
    &\leq 
    |\Om|^{-1}
    \int_0^{|\Om| t}
    (1+s)^{-\frac{3}{4}}\, ds
    (\log (e+|\Om|t))^{\frac{1}{4}}
    \\
    &\leq
    4t(1+|\Om|t)^{-\frac{3}{4}}(\log (e+|\Om|t))^{\frac{1}{4}}
    =4t\mathcal{D}_4(|\Om| t),
\end{align}
which completes the proof.
\end{proof}

\begin{proof}[Proof of Proposition \ref{prop:nonl-dec}]
    Let us decompose the nonlinear term as 
    \begin{align}
        &
        \mathcal{N}_\Om [u](t)
        =
        \underline{\mathcal{N}}_{\Omega}[u](t)
        +
        \overline{\mathcal{N}}_{\Omega}[u](t),
        \\
        &\quad 
        \underline{\mathcal{N}}_{\Omega}[u](t)
        :=
        \int_0^{\frac{t}{2}} \mathcal{T}_\Omega(t-\tau)\mathbb{P}\div (u \otimes u)(\tau)\, d\tau,
        \\
        &\quad 
        \overline{\mathcal{N}}_{\Omega}[u](t)
        :=
        \int_{\frac{t}{2}}^t \mathcal{T}_\Omega(t-\tau)\mathbb{P}\div (u \otimes u)(\tau)\, d\tau.
    \end{align}
    We see by Theorem \ref{thm:lin-decay} and Lemma \ref{lemm:GWP} that for every $2 \leq p \leq \infty$,
    \begin{align}
        \n{\nabla^m\underline{\mathcal{N}}_{\Om}[u](t)}_{L^{p}}
        &
        \leq 
        C\int_0^{\frac{t}{2}} 
        (t-\tau)^{-\frac{3}{2}(1-\frac{1}{p})-\frac{m+1}{2}}\mathcal{D}_p(|\Om|(t-\tau))
        \n{u(\tau) \otimes u(\tau)}_{L^1}\, d\tau
        \\
        &
        \leq 
        C
        t^{-\frac{3}{2}(1-\frac{1}{p})-\frac{m+1}{2}}\mathcal{D}_p(|\Om| t)
        \int_0^\infty \n{u(\tau)}_{L^2}^2\, d\tau \label{dec-L^p-1}
        \\
        &
        \leq 
        C
        t^{-\frac{3}{2}(1-\frac{1}{p})-\frac{m+1}{2}}\mathcal{D}_p(|\Om| t)
        \int_0^\infty (1+\tau)^{-\frac{3}{2}}\, d\tau
        \\
        &
        =
        C
        t^{-\frac{3}{2}(1-\frac{1}{p})-\frac{m+1}{2}}\mathcal{D}_p(|\Om| t).
    \end{align}
		
    \noindent
    {\it Step 1. The $L^p$ estimate of $\overline{\mathcal{N}}_\Om[u]$ for $2 \leq p < \infty$.}
    Based on Lemma \ref{lemm:GWP}, Proposition \ref{prop:dips-1}, the product estimate from Lemma \ref{lemm:prod-est} and the interpolation inequality from Lemma \ref{lemm:interp}, we infer that 
    \begin{align}
        \n{\nabla^m\overline{\mathcal{N}}_{\Om}[u]}_{L^p}
        &
        \leq
        C\int_\frac{t}{2}^t 
        \mathcal{D}_p(\Omega(t-\tau))
        \n{u(\tau) \otimes u(\tau)}_{\dB_{1,1}^{m+1+3(1-\frac{1}{p})}}\, d\tau
        \\
        &
        \leq
        C\int_\frac{t}{2}^t 
        \mathcal{D}_p(\Omega(t-\tau))
        \n{u(\tau)}_{L^2}
        \n{u(\tau)}_{\dot B_{2,1}^{m+1+3(1-\frac{1}{p})}}\, d\tau
        \\ 
        &
        \leq
        C\int_\frac{t}{2}^t 
        \mathcal{D}_p(\Omega(t-\tau))
        \n{u(\tau)}_{L^2}
        \n{\nabla^{m+1} u(\tau)}_{L^2}^{\frac{1}{p}}
        \n{\nabla^{m+4} u(\tau)}_{L^2}^{1-\frac{1}{p}}
        \, d\tau  
        \\
        &
        \leq
        C
        \int_\frac{t}{2}^t 
        \mathcal{D}_p(\Omega(t-\tau))
        (1+\tau)^{-\frac{3}{2}}\tau^{-\frac{m+1}{2}-\frac{3}{2}(1-\frac{1}{p})}\, d\tau  \label{dec-L^p-2}
        \\
        &
        \leq
        C
        (1+t)^{-\frac{3}{2}}
        t^{-\frac{m+1}{2}-\frac{3}{2}(1-\frac{1}{p})}
        \int_0^\frac{t}{2}
        \mathcal{D}_p(|\Omega|s)
        \, ds
        \\
        &
        \leq
        C
        (1+t)^{-\frac{3}{2}}
        t^{-\frac{m-1}{2}-\frac{3}{2}(1-\frac{1}{p})}
        \mathcal{D}_p(|\Omega| t).
    \end{align}
    Combining \eqref{dec-L^p-1} and \eqref{dec-L^p-2}, we complete the proof of \eqref{Prop:dec-p} with $2 \leq p<\infty$. 
    We emphasize that the same argument, due to the second assertion of Proposition \ref{prop:dips-1}, leads to 
    \begin{align}
      \n{\nabla^m \overline{\mathcal{N}}_{\Om}[u]}_{L^{4,\infty}}
      \leq 
      C
        (1+t)^{  -\frac{3}{2} }
        t^{-\frac{m-1}{2}-\frac{3}{2}(1-\frac{1}{4})} 
        \sp{ 1+ |\Om|t  }^{-\f{3}{4} }. 
    \end{align}
    Similarly, we get
    \begin{align}
     \n{\nabla^m \underline{\mathcal{N}}_{\Om}[u]}_{L^{4,\infty}}
      \leq 
      C
      t^{-\frac{3}{2}(1-\frac{1}{4})-\frac{m+1}{2}} \sp{ 1+ |\Om|t  }^{-\f{3}{4} }.
    \end{align}
    The two estimates above verify \eqref{Prop:dec-p-44}.  
		
    \noindent
    {\it Step 2. The $L^\infty$ estimate of  $\overline{\mathcal{N}}_\Om[u]$.}
    By Proposition \ref{prop:dips-1} and Lemmas \ref{lemm:heat} and \ref{lemm:prod-est}, it holds
    \begin{align}
        &
        \n{\nabla^m\overline{\mathcal{N}}_\Om [u](t)}_{L^\infty}
        \leq{}
        C
        \sum_{\sigma \in \{\pm\}}
        \int_{\frac{t}{2}}^t 
        \n{e^{(t-\tau)\Delta}P_{\sigma}(D)\mathbb{P}\div(u(\tau) \otimes u(\tau))}_{\dB_{2,1}^{m+\frac{3}{2}}}
        \, d\tau 
        \\
        &\quad 
        \leq{}
        C\int_{\frac{t}{2}}^t 
        (t-\tau)^{-\frac{1}{2}}\n{u(\tau) \otimes u(\tau)}_{\dB_{2,\infty}^{m+\frac{3}{2}}}
        \, d\tau 
        \\
        &\quad 
        \leq{}
        C\int_{\frac{t}{2}}^t 
        (t-\tau)^{-\frac{1}{2}}
        \n{u(\tau)}_{L^4}
        \n{u(\tau)}_{\dB_{4,\infty}^{m+\frac{3}{2}}}
        \, d\tau 
        \\
        &\quad 
        \leq{}
        C\int_{\frac{t}{2}}^t 
        (t-\tau)^{-\frac{1}{2}}
        \n{u(\tau)}_{L^4}
        \n{\nabla^{m+1} u(\tau)}_{L^4}^{\frac{1}{2}}
        \n{\nabla^{m+2} u(\tau)}_{L^4}^{\frac{1}{2}}
        \, d\tau 
        \\
        &\quad 
        \leq{}
        C\int_{\frac{t}{2}}^t 
        (t-\tau)^{-\frac{1}{2}}
        \tau^{-3-\frac{m}{2}}\mathcal{D}_4(|\Om|\tau)^2\, d\tau 
        \n{u}_{X_{\Om,4}^{0}(t)}
        \n{u}_{X_{\Om,4}^{m+1}(t)}^{\frac{1}{2}}
        \n{u}_{X_{\Om,4}^{m+2}(t)}^{\frac{1}{2}}
        \\
        &\quad 
        \leq{}
        Ct^{-\frac{5}{2}-\frac{m}{2}}\mathcal{D}_4(\Om t)^2
        \n{u}_{X_{\Om,4}^{0}(t)}
        \n{u}_{X_{\Om,4}^{m+1}(t)}^{\frac{1}{2}}
        \n{u}_{X_{\Om,4}^{m+2}(t)}^{\frac{1}{2}}.
    \end{align}
    Hence, gathering the above estimates, we complete the proof.
\end{proof}
    
Now, we are in a position to prove our main results.
\begin{proof}[Proof of Theorems \ref{main-thm-1} and \ref{thm:main-2}]
Let us first prove the decay rate of small order \eqref{non-dec-L^p}.
It follows from Theorem \ref{thm:lin-decay} and Proposition \ref{prop:nonl-dec} that
for $2 \leq p < \infty$,
\begin{align}
    \n{\nabla^m  u(t)}_{L^p} 
    \leq{}& 
    \n{ \nabla^m \mathcal{T}_\Omega(t)u_0}_{L^p}
    +
    \n{\nabla^m\mathcal{N}_\Om[u](t)}_{L^p} \\
    \leq{}& 
    \n{ \nabla^m \mathcal{T}_\Omega(t)u_0}_{L^p}
    +
    K_{p,m}
    t^{-\frac{3}{2}(1-\frac{1}{p})-\frac{m+1}{2}}
    \mathcal{D}_p(|\Om| t) \label{nonlin:p-small}\\
    ={}& 
    o\sp{t^{-\frac{3}{2}(1-\frac{1}{p})-\frac{m}{2}}
    \mathcal{D}_p(|\Om| t)}
\end{align}
as $t \to \infty$.
On the $L^\infty$ case, we see by Theorem \ref{thm:lin-decay} and Proposition \ref{prop:nonl-dec} that
\begin{align}
    \n{\nabla^m  u(t)}_{L^p} 
    \leq{}& 
    \n{ \nabla^m \mathcal{T}_\Omega(t)u_0}_{L^\infty}
    +
    \n{\nabla^m\mathcal{N}_\Om[u](t)}_{L^\infty} \\
    \leq{}& 
    \n{ \nabla^m \mathcal{T}_\Omega(t)u_0}_{L^\infty}
    + 
    K_{\infty,m}
    t^{-\frac{3}{2}-\frac{m+1}{2}}
    \mathcal{D}_\infty(|\Om| t) \\
    &
    +
    Ct^{-\frac{5}{2}-\frac{m}{2}}
    \mathcal{D}_4(|\Om| t)^2
    \n{u}_{X_{\Om,4}^{0}(t)}
    \n{u}_{X_{\Om,4}^{m+1}(t)}^{\frac{1}{2}}
    \n{u}_{X_{\Om,4}^{m+2}(t)}^{\frac{1}{2}}
    \\
    ={}& 
    o \sp{t^{-\frac{3}{2}-\frac{m}{2}}
    \mathcal{D}_\infty(|\Om| t)}
\end{align}
as $t \to \infty$,
where we have used $\n{u}_{X_{\Om,4}^{0}(t)}
\n{u}_{X_{\Om,4}^{m+1}(t)}^{\frac{1}{2}}
\n{u}_{X_{\Om,4}^{m+2}(t)}^{\frac{1}{2}}=o(1)$ ($t \to \infty$) due to \eqref{nonlin:p-small} with $p=4$.

Next, we focus on the enhanced decay estimates \eqref{enh-non-dec-L^p}, where we additionally assume that $|x|u_0(x)\in L^1(\R^3)$. In case of $2\leq p<\infty$, it follows from Theorem \ref{thm:lin-decay} and \eqref{Prop:dec-p} that 
\begin{align}
    \n{\nabla^m  u(t)}_{L^p} 
    \leq{}& 
    \n{ \nabla^m \mathcal{T}_\Omega(t)u_0}_{L^p}
    +
    \n{\nabla^m\mathcal{N}_\Om[u](t)}_{L^p} \\
    \leq{}& 
    Ct^{-\frac{3}{2}(1-\frac{1}{p})-\frac{m+1}{2}}
    \mathcal{D}_p(|\Omega| t)\n{|x|u_0(x)}_{L^1(\R^3_x)} \label{est:nonlin-Lp}
    \\
    &
    + 
     K_{p,m}
    t^{-\frac{3}{2}(1-\frac{1}{p})-\frac{m+1}{2}}
    \mathcal{D}_p(|\Om| t) \\
    \leq{}& 
    C t^{-\frac{3}{2}(1-\frac{1}{p})-\frac{m+1}{2}}
    \mathcal{D}_p(|\Om| t)
\end{align}
for all $t>0$.
It holds by Theorem \ref{thm:lin-decay}, Proposition \ref{prop:I-decay}, and \eqref{Prop:dec-p-44} that
\begin{align}
    \n{\nabla^m u(t) }_{L^{4,\infty}}
    \leq{}& 
    C \n{ \nabla^m \mathcal{T}_\Omega(t)u_0}_{L^{4,\infty}}
    +
    C\n{\nabla^m\mathcal{N}_\Om[u](t)}_{L^{4,\infty}} 
    \\ 
    \leq{}&
    Ct^{-\frac{3}{2}(1-\frac{1}{4})-\frac{m+1}{2}}(1+|\Om|t)^{-\frac{3}{4}}
    \n{ |x| u_0(x)}_{L^1(\R^3_x)} 
    \\
    &
    +  K_{4,m}
    t^{-\frac{3}{2}(1-\frac{1}{4})-\frac{m+1}{2}}
    (1+|\Om|t)^{-\frac{3}{4}} \\
    \leq{}& 
    C t^{-\frac{3}{2}(1-\frac{1}{4})-\frac{m+1}{2}}(1+|\Om|t)^{-\frac{3}{4}}
\end{align}
for all $t>0$.
Finally, we consider the decay estimates in $L^\infty$.
We have by Proposition \ref{prop:nonl-dec} that
\begin{align}
    \n{\nabla^m  u(t)}_{L^\infty} 
    \leq{}& 
    \n{ \nabla^m \mathcal{T}_\Omega(t)u_0}_{L^\infty}
    +
    \n{\nabla^m\mathcal{N}_\Om[u](t)}_{L^\infty} \\
    \leq{}& 
    Ct^{-\frac{3}{2}-\frac{m+1}{2}}
    \mathcal{D}_\infty(|\Omega| t)\n{|x|u_0(x)}_{L^1(\R^3_x)} 
    \\
    &
    + 
    K_{\infty,m}
    t^{-\frac{3}{2}-\frac{m+1}{2}}
    \mathcal{D}_\infty(|\Om| t) \label{est:nonlin-Linf}\\
    &
    +
    Ct^{-\frac{5}{2}-\frac{m}{2}}
    \mathcal{D}_4(|\Om| t)^2
    \n{u}_{X_{\Om,4}^{0}(t)}
    \n{u}_{X_{\Om,4}^{m+1}(t)}^{\frac{1}{2}}
    \n{u}_{X_{\Om,4}^{m+2}(t)}^{\frac{1}{2}}
    \\
    \leq{}& 
    C t^{-\frac{3}{2}-\frac{m+1}{2}}
    \mathcal{D}_\infty(|\Om| t)
\end{align}
for all $t \geq 1$,
where we have used 
\begin{align}
    \sup_{t>0}
    \sp{
    \n{u}_{X_{\Om,4}^{0}(t)}
    \n{u}_{X_{\Om,4}^{m+1}(t)}^{\frac{1}{2}}
    \n{u}_{X_{\Om,4}^{m+2}(t)}^{\frac{1}{2}}
    } 
    \leq C,
\end{align}
which is implied by \eqref{est:nonlin-Lp}.
Thus, we complete the proof.
\end{proof}

\begin{proof}[Proof of Theorem \ref{Thm:sharp}]
    Let $0<\varepsilon \leq 1$ and 
    \begin{align}
        u_0(x):=\varepsilon e^{-\frac{|x|^2}{4}}x^{\perp}.
    \end{align}
    Then, we see that $\n{u_0}_{L^1 \cap \dot H^{\frac{1}{2}}} \leq C\varepsilon$.
    Choosing $\varepsilon$ sufficiently small, the Fujita--Kato principle \cite{Fuj-Kat-64} implies that $u_0$ generates a global small smooth solution $u$ to \eqref{eq:nonlin} with $\Om=1$.
    From Theorem \ref{Thm:sharp-lin} and the decay estimates of nonlinear term $\mathcal{N}_\Om[u]$ obtained in the proof of Theorem \ref{main-thm-1}, it holds
    \begin{align}
        &
        \begin{aligned}
        \n{u(t)}_{L^p}
        \geq{}& \n{\mathcal{T}_\Om u_0}_{L^p} - \n{\mathcal{N}_\Om [u](t)}_{L^p}
        \\
        \geq{}& c\varepsilon t^{-\frac{3}{2}(1-\frac{1}{p})-\frac{1}{2}}\mathcal{D}_p(t) - C\varepsilon^2 t^{-\frac{3}{2}(1-\frac{1}{p})-\frac{1}{2}}\mathcal{D}_p(t) ,
        \end{aligned}
        \\
        &
        \begin{aligned}
        \n{u(t)}_{L^{4,\infty}}
        \geq{}& \n{\mathcal{T}_\Om u_0}_{L^{4,\infty}} - \n{\mathcal{N}_\Om [u](t)}_{L^{4,\infty}}
        \\
        \geq{}& 
        c\varepsilon t^{-\frac{3}{2}(1-\frac{1}{4})-\frac{1}{2}}(1+t)^{-\frac{3}{4}}
        - 
        C\varepsilon^2 t^{-\frac{3}{2}(1-\frac{1}{4})-\frac{1}{2}}(1+t)^{-\frac{3}{4}}
        \end{aligned}
    \end{align}
    for all sufficiently large $t$. Choosing $\varepsilon$ sufficiently small, we complete the proof.
\end{proof}

\vspace{10mm}

\noindent{{\bf{Acknowledgements}}} 
The work of Mikihiro Fujii was supported by JSPS KAKENHI, Grant Number JP25K17279. 
The work of Yang Li was supported by National Natural Science Foundation of China, Grant number 12571228 and Natural Science Foundation of Anhui Province, Grant number 2408085MA018. 
Jiang Xu is partially supported by National Natural Science Foundation of China, Grant Number 12271250. 
\vspace{4mm}

\noindent{{\bf{Data Availability}}} Data sharing is not applicable to this article as no datasets were generated or analysed
during the current study.

\vspace{4mm}

\noindent{{\bf{Conflicts of interest}}} All authors certify that there are no conflicts of interest for this work.

\appendix

\def\thesection{\Alph{section}}

\end{document}